\documentclass[11pt]{amsart}
\usepackage{amsmath}
\usepackage{amssymb}

\newtheorem{theorem}{Theorem}[section]
\newtheorem{lemma}[theorem]{Lemma}
\newtheorem{definition}[theorem]{Definition}
\newtheorem{example}[theorem]{Example}
\newtheorem{proposition}[theorem]{Proposition}
\newtheorem{corollary}[theorem]{Corollary}

\newtheorem{remark}{Remark}

\def \<{\langle}
\def \>{\rangle}

\def \l{\lambda }

\newcommand{\bea}{\begin{eqnarray}}
\newcommand{\eea}{\end{eqnarray}}
\newcommand{\be}{\begin {equation}}
\newcommand{\ee}{\end{equation}}
\newcommand{\g}{\mathfrak g}

\newcommand{\Z}{\mathbb Z}

\newcommand{\N}{{\mathbb Z}_{\ge 0} }
\newcommand{\C}{\mathbb C}

\newcommand{\la}{\langle}
\newcommand{\ra}{\rangle}

\newcommand{\hf}{\mbox{$\frac{1}{2}$}}

\begin{document}

\title[ Realizations of simple  affine vertex algebras and their modules ]{ Realizations of simple  affine vertex algebras and their modules: the cases $\widehat{sl(2)}$ and $\widehat{osp(1,2)}$.}
\author{Dra\v zen Adamovi\' c }
\curraddr{Faculty of Science, Department of Mathematics, University of Zagreb, Bijeni\v cka 30, 10 000
Zagreb, Croatia}
\email{adamovic@math.hr}

\date{}

\begin{abstract}
We study the embeddings of the  simple admissible  affine vertex algebras  $V_k(sl(2))$ and $V_k(osp(1,2))$, $k \notin {\mathbb Z}_{\ge 0}$,   into  the tensor product of rational  Virasoro and $N=1$ Neveu-Schwarz  vertex algebra with lattice vertex algebras.
 By  using these realizations we construct a family of weight, logarithmic and Whittaker $\widehat{sl(2)}$ and $\widehat{osp(1,2)}$--modules. As an application,  we construct  all  irreducible degenerate Whittaker modules for $V_k(sl(2))$.
\end{abstract}

\maketitle

\tableofcontents

\baselineskip=14pt
\newenvironment{demo}[1]{\vskip-\lastskip\medskip\noindent{\em#1.}\enspace
}{\qed\par\medskip}


\section{Introduction}

Let $V^k(\g)$ denotes the universal affine vertex algebra of level $k$ associated to a simple finite-dimensional Lie super algebra $\g$. Let  $J^k(\g)$ be the maximal ideal in $V^k(\g)$ and $V_k(\g) = V^k(\g) / J^k(\g)$   its simple quotient. The representation theory of $V_k(\g)$   depends on  the structure of  the maximal ideal  $J^k(\g)$.
 One sees that a  $V^k(\g)$--module $M$  is  a  module for the simple vertex algebra $V_k(\g)$ if and only if  $J^k (\g) . M = 0$.
Such approach   can be applied for a construction and classification of modules in the category $\mathcal O$ and in the category of weight modules (cf. \cite{AdM-MRL}, \cite{Ara2}, \cite{AFR},  \cite{CR-2013}, \cite{RSW},   \cite{RW2}). But  it seems that  for a construction of logarithmic, indecomposable and Whittaker modules one needs different methods.

In this paper we explore the possibility that a simple affine vertex algebra can be realized as a vertex  subalgebra of the tensor product:
\bea \label{inclusion}  V_k(\g) \subset W(\g) \otimes \Pi_{\g} (0)\eea where $W(\g)$ is a $\mathcal W$-algebra associated to $\g$  and $\Pi_{\g}(0)$ is a lattice type vertex algebra. This can be treated as an inverse of the quantum Hamiltonian reduction (cf. \cite{Sem-inv}).

In this moment we can not prove that such  inclusion exists in general, but we  present a proof of  (\ref{inclusion})  in  the cases  $\g=sl(2)$ and $\g= osp(1,2)$. Let us describe our results in more details.
Let $V^{Vir} (d_{p,p'}, 0)$  and  $V^{ns} (c_{p,q}, 0)$ denote the universal  Virasoro and N=1 Neveu-Schwarz vertex algebras with central charges:
$d_{p,p'} = 1 - \frac{6 (p-p') ^2 }{ p p' }$ and $c_{p,q} = 3/2 - \frac{3 (p-q) ^2 }{p q}$. Their simple quotients are denoted by   $L^{Vir} (d_{p,p'}, 0)$ and  $L^{ns}(c_{p,q}, 0)$.  Let $\Pi(0)= M(1) \otimes {\C} [{\Z} c] $  and  $\Pi ^{1/2}(0)= M(1) \otimes {\C} [{\Z} \frac{c}{2} ] $  be the  vertex algebras of lattice type associated to the lattice of
$L= {\Z} c + {\Z} d$, with products $$ \langle  c, c\rangle = \langle d, d \rangle = 0, \quad \langle c, d \rangle = 2. $$
Let $F$ be the fermionic vertex algebra of central charge $c=1/2$ associated to a neutral fermion  field.

We prove:

\begin{theorem}
There are non-trivial homomorphisms of  simple admissible affine  vertex algebras:
 \item [(1)] $ \Phi_1  : V_k(sl(2) ) \rightarrow L^{Vir} (d_{p,p'}, 0) \otimes \Pi(0)$ where $k + 2 = \frac{p}{p'}$ such that $p, p' \ge 2$, $(p,p') =1$,

 \item[(2)] $\Phi _2  : V_k(osp(1,2) ) \rightarrow L^{ns} (c_{p,q}, 0) \otimes F \otimes  \Pi^{1/2} (0), $ where  $k + 3/2 = \frac{p}{2q}$, such that $p, q \in {\Z}$, $p, q \ge 2$, $(\frac{p-q}{2}, q) = 1$.
\end{theorem}

Let us discuss some application of previous theorem in the case $V_k(sl(2))$:
\begin{itemize}
\item We show in Section \ref{relaxed-realization} that all relaxed highest weight modules for the admissible vertex algebra  $V_k(sl(2))$ have the form
$$L^{Vir}(d_{p,p'},h ) \otimes \Pi_{(-1)} (\lambda) $$
where $L^{Vir}(d_{p,p'},h ) $ is an irreducible $L^{Vir} (d_{p,p'}, 0)$--module and $ \Pi_{(-1)} (\lambda) $ is a weight $\Pi(0)$--module. These modules were first detected in \cite{AdM-MRL} by using the theory of Zhu's algebras. We also show that the character of $L^{Vir}(d_{p,p'},h ) \otimes \Pi_{-1} (\lambda) $ coincides with the Creutzig-Ridout  character formula presented in \cite{CR-2013} and proved recently in \cite{KR}.

We should also say that a similar realization of  irreducible relaxed highest weight modules were presented in \cite[Section 9]{A-2007} in the case of critical level  for $A_1^{(1)}$ and in \cite[Corollary 7]{A-2014} in the case of affine Lie algebra $A_2 ^{(1)}$ at level $k =-3/2$.

\item We prove in Section \ref{Whittaker} that a family of degenerate Whittaker modules for $V_k(sl(2) )$ have the form
$$L^{Vir}(d_{p,p'},h ) \otimes \Pi_{\lambda} $$
where $\Pi_{\lambda}$ is a Whittaker $\Pi(0)$--modules. This result is the final step in the classification and realization of Whittaker $A_1^{(1)}$--modules (all other Whittaker $A_1^{(1)}$--modules were realized in \cite{ALZ-2016}). But our present result implies that affine admissible vertex algebra $V_k(sl(2))$ admits a family of Whittaker modules.  On can expect a similar result in general.

\item     In Section  \ref{logarithmic} we present a vertex-algebraic  construction of logarithmic modules by using the methods from \cite{AdM-selecta} and the expressions for screening operators from \cite[Section 5]{FFHST}.  We prove that the admissible vertex algebra $V_k(sl(2))$,  for arbitrary admissible  $k \notin {\Z}_{\ge 0}$,  admits logarithmic modules $\widetilde{ \mathcal M_{r,s} ^{\ell, \pm} (\lambda) }$  of nilpotent rank two (cf.  Corollaries \ref{cor-log-1}  and  \ref{cor-log-2}).   These logarithmic modules   were previously constructed  only for   levels $k=-1/2$ and $k=-4/3$ (cf. \cite{LMRS}, \cite{Ridout11}, \cite{AdM-selecta}, \cite{Gab}).

\item We present in Section \ref{real-N3} a realization of the simple  affine vertex algebra  $\mathcal W_{k'} (spo(2,3), f_{\theta})$  with central charge $c=-3/2$. It  is realized on the tensor product of the simple super-triplet vertex algebra $SW(1)$ (introduced  by the author and A. Milas in \cite{AdM-supertriplet})  and a rank one lattice vertex algebra. As a consequence we give a  direct proof that the parafermion vertex algebra $K(sl(2), -\frac{2}{3})$ is a ${\Z}_2$--orbifold of  a  super-singlet vertex algebra, also   introduced in \cite{AdM-supertriplet}.
We should  mention  that a different approach based on the  extension theory was  recently presented in \cite{ACR}.
\end{itemize}

Some applications in the case  $V_k(osp(1,2) )$ will be presented in our forthcoming paper \cite{A-new}. Let us note here that we have the following realization at the critical level.
 We introduce a  vertex algebra $ V^{ns} _{crit}  $ which is  freely generated by $G^{crit}$ and $T$, such that $T$ is central and the following $\lambda$--bracket  relation holds:
$$ [G ^{cri} _{\lambda} G^{cri} ] =  2 T  -  \lambda ^2. $$
We prove:

\begin{theorem}
Let $k=-3/2$. There is   non-trivial homomorphism of   vertex algebras:
$$\overline \Phi  : V^k(osp(1,2))  \rightarrow  V^{ns} _{crit}    \otimes F \otimes     \Pi^{1/2} (0). $$
such that $T$ is a central element of $V^k(osp(1,2))$.
\end{theorem}

In Remarks   \ref{ired-comment}, \ref{KR-sl} and  \ref{generalization-osp(1,2)}  we discuss how one prove irreducibility of modules of the type
\begin{itemize}
\item
$ L^{Vir} (d_{p,p'}, h) \otimes \Pi_{(-1)}(\lambda) $  for $V_k(sl(2))$,
\item $ L^{\mathcal R} (c_{p,q}, h) ^{\pm} \otimes M^{\pm} \otimes \Pi_{(-1)} ^{ 1/2} (\lambda)$ for $V_k(osp(1,2))$,
\end{itemize}
by using new results on characters of relaxed  highest weight modules from  \cite{KR}.

 In our forthcoming papers we plan to investigate a higher rank generalizations of the result discussed above.

\vskip 5mm
{\bf Acknowledgment.}
 The author is partially supported by the Croatian Science Foundation under the project 2634 and by the
QuantiXLie Centre of Excellence, a project cofinanced
by the Croatian Government and European Union
through the European Regional Development Fund - the
Competitiveness and Cohesion Operational Programme
(KK.01.1.1.01).
  We  would like to thank T. Creutzig,  A. Milas,   G. Radobolja, D. Ridout and S. Wood on valuable discussions.
Finally we thank the referee on careful reading of the paper and helpful  comments.

\section{Preliminaries}
\label{prelimin}
In the paper we assume that the reader is familiar with basic concepts in the vertex algebra theory such as modules, intertwining operators.
In this section we recall definition of the logarithmic modules for vertex operator algebras and construction of logarithmic modules. We also recall how we can extend vertex operator algebra $V$ by its module $M$ and get extended vertex algebra $\mathcal V = V \oplus M$, and construct $\mathcal V$--modules. This construction is important for the construction of logarithmic modules for affine vertex algebra $V_k(sl(2))$ in Section \ref{logarithmic}.

\subsection{Logarithmic modules} Let us   recall the  definition of  the  logarithmic module of a vertex operator algebra. More informations on the theory of logarithmic modules for vertex operator algebras can be found in the papers   \cite{AdM-2013}, \cite{CG},  \cite{HLL},  \cite{Mil}, \cite{Miy}.

Let $(V,Y,{\bf 1}, \omega)$ be a vertex operator algebra, and $(M, Y_M)$ be its weak module. Then  the components of the field
$$ Y_M(\omega, z) = \sum_{n \in {\Z}} L(n) z^{-n-2} $$
defines on $M$ the structure of a module for the Virasoro algebra.
\begin{definition}
A weak module     $(M, Y_M)$  for the vertex operator algebra $(V, Y , 1, {\omega})$  is called  a logarithmic module if it admits the following decomposition
   $$ M = \coprod_{r \in {\C}}  M_r, \quad  M_r = \{ v \in M \ \vert \ (L(0) -r) ^k = 0 \ \mbox{for some} \ k \in {\Z}_{>0} \}. $$

  If $M$ is a logarithmic module, we say that it has a  nilpotent  rank $m \in {\Z_{\ge 1} }$ if $$ (L(0)-L_{ss} (0) ) ^{m}  = 0, \quad (L(0)-L_{ss} (0) ) ^{m-1}   \ne 0,$$
where $L_{ss}(0)$ is the semisimple part of $L(0)$.
\end{definition}

Now we shall recall the construction of logarithmic modules from \cite{AdM-selecta}.
 Let $v \in
V$ such that
\bea
&& [v_n, v_m] = 0 \quad \forall n,m \in {\Z}, \label{rel-c-1} \\
&& L(n) v = \delta_{n,0} v \quad \forall \ n \in {\N},
\label{rel-c-2}
\eea
so that $v$ is of conformal weight one.

Define
$$ \Delta(v,z) = z^{v_0} \exp \left( \sum_{n=1} ^{\infty}
\frac{v_n}{-n}(-x)^{-n} \right). $$

The following result was proved in  \cite[Theorems 2.1  and 2.2]{AdM-selecta}
\begin{theorem} \cite{AdM-selecta} \label{AdM-selecta}
\label{gen-const-log} Assume that $V$ is a vertex operator
algebra and that $v \in V$ satisfies conditions
(\ref{rel-c-1}) and (\ref{rel-c-2}). Let $\overline{V}$ be the
vertex subalgebra of $V$ such that $\overline{V} \subseteq
\mbox{Ker}_V v_0$.

\item[(1)]
Assume that $(M,Y_M)$ is
a weak    $V$--module.
Define the pair $(\widetilde{M}, \widetilde{Y}_{\widetilde{M} })$
such that
$$\widetilde{M} = M \quad \mbox{as a vector space}, $$
$$ \widetilde{Y}_{\widetilde{M} } (a, x) = Y_{M} (
\Delta(v,x) a, x) \quad \mbox{for} \ a \in \overline{V}. $$
 Then  $(\widetilde{M},
\widetilde{Y}_{\widetilde{M} } )$ is a weak $
\overline{V}$--module.
\item[(2)] Assume that $(M,Y_M)$ is a $V$--module such that $L(0)$ acts
semisimply on $M$. Then $(\widetilde{M},
\widetilde{Y}_{\widetilde{M} })$ is logarithmic
$\overline{V}$--module if and only if $v_0$ does not act
semisimply on $M$. On   $\widetilde{M}$ we have:
$$ \widetilde{ L(0)} = L(0) + v_0$$
\end{theorem}

\subsection{ Extended vertex algebra $\mathcal V = V \oplus M$}
\label{extended}
Let $(V, Y_V , {\bf 1}, \omega)$ be a vertex operator algebra and $(M, Y_M)$ a $V$--module having integral weights with respect to $L(0) $, where $L(n) = \omega_{n+1}$. Let $\mathcal V = V\oplus M$.
Define $$ Y_{\mathcal V} ( v_1 + w_1, z) ( v_2 + w_2) = Y_V  (v_1, z) v_2 +  Y _M (v_1, z ) w_2  + e^{ z L(-1) } Y_M( v_2, -z ) w_1, $$
where $v_1, v_2  \in V$, $w_1, w_2 \in M$.
Then by \cite{Li-1994} $(\mathcal V,  Y_{\mathcal V} ,  {\bf 1}, \omega )$ is a vertex operator algebra.
The following lemma gives a method for a construction of a family of  $\mathcal V$--modules.

\begin{lemma}\cite{AdM-2012} \label{AdM-2012}
Assume that $(M_2, Y_{M_2}) $ and $(M_3, Y_{M_3})$ be $V$--modules, and let $\mathcal Y (\cdot, z)$ be an intertwining operator of type   $\binom{ M_3}{ M  \ \ M_2}$ with integral powers of  $z$. Then $(M_2 \oplus M_3, Y_{M_2 \oplus M_3})$ is a $\mathcal V$--module, where the vertex operator is  given by
$$ Y_{M_2 \oplus M_3} (v + w) ( w_2 + w_3) = Y_{M_2} (v, z) w_2 + Y_{M_3} (v, z) w_3   + \mathcal Y (w, z) w_2 $$
 for $v \in V$, $w \in M$, $w_i \in M_i$, $i=1,2$.
\end{lemma}

\begin{remark}
Note that the vertex operator algebra $\mathcal V$ is not simple, and that the module $M_2 \oplus M_3$ is also not simple. Moreover, the module structure on  $M_2 \oplus M_3$ is, in general,  not unique.
\end{remark}

\subsection{Fusion rules for the minimal Virasoro vertex operator algebras}
\label{Vir-subsect}
Now we review some results on Virasoro vertex operator algebras, their fusion rules and intertwining operators. More details can be found in \cite{FZ}, \cite{Wa}.

  Let $V^{Vir}(c,0)$ be the universal Virasoro vertex operator algebra of central charge $c$, and  let $L^{Vir} (c,0)$ be it simple quotient (cf. \cite{FZ}).
 Let
 $$ d_{p,p'} = 1- 6 \frac{(p-p') ^2}{p p'}, \quad p, p' \in {\Z}_{\ge 2}, \ (p,p') = 1 $$
 The Virasoro vertex algebra  $L^{Vir} (d_{p,p'},0)$ is rational (cf. \cite{Wa}) and its  irreducible modules are
 $\{ L^{Vir} (d_{p,p'}, h) \  \vert \  h \in   \mathcal{S}_{p,p'}\
\} $ where     $$ \mathcal{S}_{p,p'}= \{ h_{p,p'} ^{r,s }  = \frac{ ( s  p - r p')^2  - (p-p ')^2 } {4 p p'}  \ \vert \  \ 1 \le r \le p-1, 1 \le s \le p'-1\}. $$
Let us denote the highest weight vector in $L^{Vir} (d_{p,p'},  h_{p,p'} ^{r,s } )$ by $v_{r,s}$.

The fusion rules for  $L^{Vir} (d_{p,p'},0)$--modules are
$$ L^{Vir} (d_{p,p'},  h_{p,p'} ^{r,s } )  \boxtimes   L^{Vir} (d_{p,p'},  h_{p,p'} ^{r',s' } ) = \sum_{ r'', s''}   \left[\begin{matrix} (r'', s'') \\ (r, s) \ (r', s')\end{matrix}\right]_{ (p,p') }  L^{Vir} (d_{p,p'},  h_{p,p'} ^{r'',s'' } ), $$
where the fusion coefficient   $\left[\begin{matrix} (r'', s'') \\ (r, s) \ (r', s')\end{matrix}\right]_{ (p,p') }$  is equal to the dimension of the vector space of all intertwining operators of the type
$$   \binom{L^{Vir} (d_{p,p'},  h_{p,p'} ^{r'',s'' })}{L^{Vir} (d_{p,p'},  h_{p,p'} ^{r,s } )  \ \   L^{Vir} (d_{p,p'},  h_{p,p'} ^{r',s' } ) }.  $$
The fusion coefficient $\left[\begin{matrix} (r'', s'') \\ (r, s) \ (r', s')\end{matrix}\right]_{ (p,p') }$ is $0$  or $1$. For explicit formula see \cite{Wa}, \cite[Section 7]{CHY}.

Let $\mathcal Y(\cdot, z)$ be a non-trivial intertwining operator of type $$   \binom{L^{Vir} (d_{p,p'},  h_3)}{L^{Vir} (d_{p,p'},  h_1 )  \ \   L^{Vir} (d_{p,p'},  h_2) }, \quad (h_i \in S_{p,p'}).  $$ Then
for every $v \in  L^{Vir} (d_{p,p'},  h_1 )$ we have
$$ \mathcal Y(v, z) = \sum_{ r \in \Delta + {\Z}} v_r z^{-r-1} $$
where $\Delta = h_1 + h_2 - h_3$. Let $v_{h_i}$ be the highest weight vector in $L^{Vir} (d_{p,p'},  h_i)$, $i=1,2,3$. Then one can show that
$$ (v_{h_1} )_{\Delta -1} v_{h_2} = C v_{h_3}, \quad  (v_{h_1} )_{\Delta +n } v_{h_2} = 0$$
where $C \ne 0$, $n \in {\Z}_{\ge 0}$.

\section{Wakimoto modules for $\widehat{sl(2)}$}

In this section we  first recall the construction of the Wakimoto modules for  $\widehat{sl(2)}$ (cf. \cite{efren}, \cite{Wak}).  Then by using  the  embedding of  the Weyl vertex algebra into  a lattice vertex algebra (also called FMS bosonization)  we show that  the  universal affine vertex algebra $V^k(sl(2))$   can be embedded  into the  tensor product of a Virasoro vertex algebra with a vertex algebra $\Pi(0)$ of a lattice type. This result is stated in Proposition  \ref{prop-sl2-1},   which is a vertex-algebraic interpretation  of the result of A. M. Semikhatov from \cite{Sem-1994}.

\vskip 5mm

\subsection{Weyl vertex algebra $W$} Recall that the   Weyl algebra  {\it Weyl} is an associative algebra with generators
$ a(n), a^{*} (n) \quad (n \in {\Z})$ and relations
\bea  \label{comut-Weyl} && [a(n), a^{*} (m)] = \delta_{n+m,0}, \quad [a(n), a(m)] = [a ^{*} (m), a ^* (n) ] = 0 \quad (n,m \in {\Z}). \eea
Let $W$ denotes the simple {\it Weyl}--module generated by the cyclic vector ${\bf 1}$ such that
$$ a(n) {\bf 1} = a  ^* (n+1) {\bf 1} = 0 \quad (n \ge 0). $$
As a vector space $ W \cong {\C}[a(-n), a^*(-m) \ \vert \ n >0 , \ m \ge 0 ]. $
There is a  unique vertex algebra $(W, Y, {\bf 1})$ where
the  vertex operator map is $ Y: W  \rightarrow \mbox{End} (W )[[z, z ^{-1}]] $
such that
$$ Y (a(-1) {\bf 1}, z) = a(z), \quad Y(a^* (0) {\bf 1}, z) = a ^* (z),$$
$$ a(z)   = \sum_{n \in {\Z} } a(n) z^{-n-1}, \ \ a^{*}(z) =  \sum_{n \in {\Z} } a^{*}(n)
z^{-n}. $$

\subsection{ The Heisenberg vertex algebra $M_{\delta}( \kappa, 0)$. }
Let ${\mathfrak h} = {\C}  \delta   $ be the $1$--dimensional commutative   Lie algebra with a symmetric bilinear form defined by $( \delta   , \delta  ) =1  $, and  $\widehat{\mathfrak h} = {\mathfrak h}  \otimes {\C}[t,t ^{-1}] + {\C} c$ be its affinization.
 Set $\delta   (n) = \delta  \otimes t ^n $.  
 Let $M_{\delta}( \kappa ,  0 )$  denotes the simple $\widehat{\mathfrak h}$--module of level $\kappa \ne 0 $ generated by the  vector ${\bf 1}$ such that
 $ \delta (n) {\bf 1} = 0 \quad \forall n \ge 0. $
 As a vector space
$ M_{\delta}( \kappa , 0 )= {\C}[\delta (n) \,|\,  n \le -1 ]$.

There is a  unique   vertex algebra  $(M_{\delta} (\kappa, 0), Y , {\bf 1} )$ generated  by the  Heisenberg field $Y( \delta(-1) {\bf 1}, z) = \delta (z) = \sum_{n \in {\Z}} \delta (n) z ^{-n-1} $ such that
$$[\delta  (n), \delta (m) ] = \kappa  n \delta_{n+m,0}  \quad (n,m \in {\Z}). $$ Vector
$\omega =  \left( \frac{1}{2 \kappa} \delta(-1) ^2 + a \delta(-2) \right){\bf 1}$
is a conformal vector of central charge $1- \frac{12 a ^2 } {\kappa}. $

For $r  \in {\C}$, let $M_{\delta} (\kappa , r  )$ denotes the irreducible  $M_{\delta}( \kappa , 0 )$--module generated by the highest  weight vector $v_{r}$ such that
$$ \delta(n) v_{r  } = r  \delta_{n,0} v_{r} \quad (n \ge 0). $$

We can consider  lattice  $D_{r} = {\Z} ( \frac{ r \delta }{  \kappa } )$ and  the generalized lattice vertex algebra $V_{ D_{r }  }:= M_{\delta} (\kappa , 0) \otimes {\C} [ D_{r} ] $ (cf. \cite{DL}).    We have:
$$ M_{\delta}( \kappa , r ) = M_{\delta}( \kappa , 0). e^{  \frac{ r \delta  }{ \kappa }  }. $$

Then the restriction of the  vertex operator $Y( e^{ \frac{r \delta}{\kappa } }, z)$ on   $M_{\delta}( \kappa , 0)$ can be considered  as a  map
$ M_{\delta} (\kappa, 0 )  \rightarrow  M_{\delta} (\kappa, r )[[ z, z^{-1} ]] $.

\subsection{ The Wakimoto module  $W_{k,\mu}$. }
Assume that $k \ne -2$ and  $\mu \in {\C}$.
Let $$W_{k,\mu}= W \otimes  M_{\delta} (2 (k+2), \mu). $$  Then $W_{k,0}$ has the structure of a vertex algebra and $W_{k, \mu}$ is a $W_{k,0}$--module.

Let $V^ k(sl(2))$ be the universal vertex algebra of level $k$ associated to the affine Lie algebra $\widehat{sl(2)}$. There is a injective homomorphism of vertex algebras
$\Phi : V^ {k}(sl(2)) \rightarrow W_{k,0}$ generated by
\bea
e(z) &=& a(z); \nonumber \\
h(z) & =& -2 : a^{*}(z)  a (z) : + \delta(z); \nonumber \\
f(z) & = & - : a^{*} (z) ^2 a (z) : + k \partial_z a^{*} (z) + a^{*}(z) \delta (z). \nonumber
\eea
The screening operator is $Q = \mbox{Res}_z : a(z) Y(  e ^{-\tfrac{1}{k+2} \delta}, z) : =   (a(-1) e ^{-\tfrac{1}{k+2} \delta} )_0$ (cf. \cite{efren}).

\subsection{Bosonization} Let $H$ be the lattice
$$ H= {\Z} \alpha + {\Z}\beta, \ \la \alpha , \alpha \ra = - \la \beta , \beta \ra = 1, \quad \la \alpha, \beta \ra = 0, $$
and $V_H = M_{\alpha, \beta} (1) \otimes {\C} [L]$ the associated lattice vertex algebra, where $M_{\alpha, \beta} (1) $ denotes the Heisenberg vertex algebra generated by $\alpha$ and $\beta$.


The Weyl vertex algebra  $W$ can be realized as a subalgebra of $V_{H}$ generated by
$$ a = e^{\alpha + \beta}, \ a^{*} = -\alpha(-1) e^{-\alpha-\beta}. $$
This gives a realization of the universal affine vertex algebra $V_k(sl(2))$ as a subalgebra of
$V_H \otimes M_{\delta} (2 (k+2), 0)$
generated by
\bea
e & = & e^{\alpha+ \beta}, \label{def-e-1} \\
h & = & - 2 \beta(-1) + \delta(-1) \label{def-h-1} \\
f & = & \left[ (k+1) ( \alpha(-1) ^2 - \alpha(-2) ) + (k+2) \alpha (-1) \beta(-1) - \alpha(-1) \delta(-1) \right] e^{-\alpha-\beta} \label{def-f-1}.
\eea
Screening operators are (cf. \cite[Section 7]{efren}):
\bea
\label{scr}   Q = Res_z ( Y( e^{\alpha + \beta - \tfrac{1}{k+2} \delta}, z),    \quad  \widetilde{Q} = \mbox{Res}_z Y( e^{ - (k+2) (\alpha + \beta ) +\delta}, z) .
\eea

\vskip 5mm

\subsection {Embedding of $V^k(sl(2))$ into vertex algebra $V^{Vir} (d_k, 0) \otimes \Pi(0)$.}  We shall first define new generators of the Heisenberg vertex algebra $M_{\alpha, \beta} (1) \otimes M_{\delta} (2 (k+2)) $.  Let
$$ \gamma = \alpha + \beta - \tfrac{1}{k+2}\delta, \quad \mu = - \beta + \tfrac{1}{2} \delta, \quad \nu = -\tfrac{k}{2} \alpha - \tfrac{k+2}{2}\beta  + \tfrac{1}{2}\delta. $$
Then
$$ \langle \gamma, \gamma \rangle = \frac{2}{k+2}, \quad \langle \mu, \mu \rangle = - \la \nu, \nu \ra = \frac{k}{2}, $$
and all other products are zero. For our calculation, it is useful to notice that
\bea
 && \alpha = \nu + \tfrac{k+2}{2}\gamma, \nonumber \\
 &&  \beta = -\tfrac{k+2}{2} \gamma + \tfrac{2}{k} \mu - \tfrac{k+2}{k} \nu, \nonumber \\
 && \delta = - (k+2) \gamma + \tfrac{2 (k+2)}{k} \mu - \tfrac{2 (k+2)}{k} \nu. \nonumber
 \eea

 Let $M(1) := M_{\mu, \nu} (1) $  be the Heisenberg vertex algebra generated by $\mu$ and $\nu$. Consider the rank one lattice $\Z c \subset M(1)$ where $c =\tfrac{2}{k} (\mu - \nu).$ Then
  $$\Pi(0):=  M (1) \otimes {\C}[{\Z} c ] $$
  has the structure of a vertex algebra. Some properties of $\Pi(0)$ will be discussed in Section  \ref{pi0}.

  Let $M_{\gamma}(\frac{2}{k+2} )$ be the Heisenberg vertex algebra generated by $\gamma$.
We obtain the following expression for  the generators of $V^ k(sl(2))$:
\bea
\ \ e & = & e^{\tfrac{2}{k} (\mu - \nu) }, \label{def-e-2} \\
\ \ h & = & 2 \mu(-1)\label{def-h-2} \\
\ \ f & = & \left[ \tfrac{1}{4} (k+2) ^{2}  \gamma(-1) ^2  - \tfrac{1}{2} (k+1) (k+2) \gamma(-2) -\nu(-1)^{2} - (k+1) \nu(-2) \right] e^{-\tfrac{2}{k} (\mu - \nu)} \label{def-f-2}.
\eea
Set
$$ \omega ^{(k)}  =  \left(\frac{k+2}{4} \gamma (-1) ^2 - \frac{k+1}{2}\gamma (-2)\right){\bf 1}. $$
Then
$$   f =  \left[  (k+2) \omega^{(k)}  -\nu(-1)^{2} - (k+1) \nu(-2) \right] e^{-\tfrac{2}{k} (\mu - \nu)}.$$
Note that  $  \omega^{(k)} $  generates the universal Virasoro vertex algebra $V^{Vir} (d_k, 0)$ where
$ d_k = 1 - 6 \frac{(k+1)^2}{ (k+2)}$,
 which is realized as a subalgebra of the Heisenberg vertex algebra $ M_{\gamma} (\tfrac{2}{k+2} ,0 )$.

 As usual we set $L(n) = \omega_{n+1}$  and  denote the Virasoro field by $L(z) = \sum_{n\in {\Z}} L(n) z^{-n-2}$.

We get the following result:
\begin{proposition} \cite{Sem-1994}  \label{prop-sl2-1}
Let $\omega$ be  the conformal vector in $V^{Vir} (d_k, 0)$. There is a injective homomorphism of vertex algebras
$$\Phi : V^{k} (sl(2)) \rightarrow V^{Vir} (d_k, 0) \otimes \Pi(0) \subset  M_{\gamma} (\tfrac{2}{k+2} ,0 ) \otimes \Pi(0)  $$
such that
 \bea
e & \mapsto & e^{\tfrac{2}{k} (\mu - \nu) }, \label{def-e-3} \\
h & \mapsto & 2 \mu(-1), \label{def-h-3} \\
f & \mapsto & \left[     (k+2) \omega   -\nu(-1)^{2} -  (k+1) \nu(-2) \right] e^{-\tfrac{2}{k} (\mu - \nu)} \label{def-f-3}.
\eea
\end{proposition}

\begin{remark}
The realization in Proposition \ref{prop-sl2-1} had first obtained by A. M. Semikhatov  in \cite{Sem-1994} using slightly different notations.
\end{remark}

A critical level version of this proposition was obtained in \cite{ALZ-2016}. Let $M_T(0) ={\C} [T(-n), n \ge 2]$ be the commutative vertex algebra generated by the commutative  field
$$T(z) = \sum_{n \le -2} T(n) z ^{-n-2}.$$

\begin{proposition}  \label{prop-sl2--crit}
Let $k=-2$. There is a injective homomorphism of vertex algebras
$$\Phi : V^{k} (sl(2)) \rightarrow M_T(0) \otimes \Pi(0)  $$
such that
 \bea
e & \mapsto & e^{\tfrac{2}{k} (\mu - \nu) }, \label{def-e-3-crit} \\
h & \mapsto & 2 \mu(-1), \label{def-h-3-crit} \\
f & \mapsto & \left[  T(-2)   -\nu(-1)^{2} -  (k+1) \nu(-2) \right] e^{-\tfrac{2}{k} (\mu - \nu)} \label{def-f-3-crit}
\eea
\end{proposition}

 \section{Some $\Pi(0)$--modules.}
 \label{pi0}
 In this section we study vertex algebra $\Pi(0)$ which is associated to  an  isotropic rank two  lattice   $L= {\Z} c + {\Z} d $.

 Lattice  $L$  is realized as $L= {\Z} c + {\Z} d \subset M(1)$, where $c = \alpha + \beta = \frac{2}{k} (\mu -\nu)$, and  $d = \mu + \nu$. Then $$\langle c, c \rangle = \langle d, d \rangle = 0, \quad \langle c, d \rangle =2. $$
 The vertex algebra $\Pi(0) = M(1) \otimes {\C}[\Z c]$ is generated by $c(-1) ,d(-1) , u = e^c, u^{-1} = e^{-c}. $ Its representation theory was studied in \cite{DBT}.

 \subsection{ Weight   $\Pi(0)$--modules and their characters}  Let us recall some steps in the construction of $\Pi(0)$--modules.
 Let $\mathcal A$ be the associative algebra generated by $d, e^{n c}$,  where $n \in {\Z}$ and relations
 $$ [ d, e^{n c} ] = 2 n e^{ n c}, \quad e^{n c} e^{m c} = e ^{ (n+m) c}, \quad (n, m \in {\Z}). $$
 (We use  the convention $ e^0 = 1$). By using results from  \cite[Section 4]{DBT}  we see that for any $\mathcal A$--module $U$ and any $ r \in {\Z}$, there exists a unique  $\Pi(0)$--module  structure  on the vector space
 $$\mathcal L_{r} (U) =  U \otimes M(1)$$  such that  $c(0) \equiv r \mbox{Id}$ on $\mathcal L_{r} (U)$. Moreover $\mathcal L_{r} (U)$ is irreducible $\Pi(0)$--module  if and only if $U$ is irreducible $\mathcal A$--module.
 By using this method one can construct  the  weight $\Pi(0)$--modules from \cite{DBT}.
 (see Proposition \ref{constr-pi0-mod} below).

In the present paper  we shall need  the following simple current extension of $\Pi(0)$:

$$\Pi ^{1/2} (0) = M(1) \otimes {\C} [{\Z} \tfrac{c}{2}]  = \Pi (0) \oplus \Pi(0) e^{\frac{c}{2}}. $$
$\Pi ^{1/2} (0) $ is again the vertex algebra of the same type and
it is generated by $c(-1) ,d(-1) ,  u^{1/2}=   e^{c/2}, u^{-1/2} = e^{-c/2}. $ Note that $g = \exp [\pi i \mu(0)]$ is an automorphism  of order two of the vertex algebra $\Pi ^{1/2} (0) $ and that $g= \mbox{Id}$ on $\Pi (0)$.

In order to construct $\Pi ^{1/2} (0)$--modules we need to consider a slightly larger associative algebra.  Let $\mathcal A^{1/2}$ be the associative algebra generated by $d, e^{n c}$,  $n \in \tfrac{1}{2}{\Z}$, and relations
 $$ [ d, e^{n c} ] = 2 n e^{ n c}, \quad e^{n c} e^{m c} = e ^{ (n+m) c} \quad (n, m \in \frac{1}{2} {\Z}). $$
For any $\mathcal A^{1/2} $--module $U'  $ and any $ r \in {\Z}$, there exists a unique  ($g$--twisted)  $\Pi^{1/2} (0)$--module  structure  on the vector space
 $$\mathcal L_{r} (U' ) =  U'  \otimes M(1)$$  such that  $c(0) \equiv r \mbox{Id}$ on $\mathcal L_{r} (U')$. Module $\mathcal L_{r} (U')$ is untwisted if $r$ is even and  $g$--twisted if $r$ is odd. We omit details, since arguments  are completely analogous to those of \cite{DBT}.
 In this way we get a realization of a family of irreducible modules for the vertex algebras  $\Pi(0)$ and $\Pi^{1/2} (0)$.
%
 \begin{proposition} \label{constr-pi0-mod}
\item[(1)]  \cite{DBT} For every $r \in {\Z}$ and $\lambda \in {\C} $, $\Pi _{(r)}  (\lambda) :=\Pi(0) e^{r \mu  + \lambda c}$ is an irreducible $\Pi(0)$--module on which $c(0)$ acts as
 $r \mbox{Id}$.

 \item[(2)]  Assume that  $r \in {\Z}$  is even (resp. odd) and $\lambda \in {\C} $. Then  $\Pi  ^{1/2} _{(r)}  (\lambda) :=\Pi ^{1/2} (0) e^{r \mu  + \lambda c}$ is an irreducible untwisted (resp. $g$--twisted)  $\Pi ^{1/2} (0)$--module on which $c(0)$ acts as
 $r \mbox{Id}$.

 \end{proposition}

As usual for a vector $V$ is a vertex algebra $V$ we define
 $$ \Delta(v,z) = z^{v_0} \exp \left(  \sum_{n =1} ^{\infty} \frac{v_n}{-n} (-z ) ^{-n} \right). $$

 The following lemma follows from \cite[Proposition 3.4]{Li}.
 \begin{lemma} \label{Li} For $\ell, r \in {\Z} $ we have
 $$ (\Pi_{(\ell + r) } (\lambda) , Y_{ \Pi_{(\ell+ r) } (\lambda)  } (\cdot, z) ) \cong   ( \Pi_{(r) } (\lambda)  , Y_{ \Pi_{(r)  } (\lambda)  } ( \Delta (\ell \mu ,z) \cdot, z)  ). $$
 \end{lemma}

 We also have the following  important observation which essentially  follows from the analysis of $\Pi _{(-1)}  (\lambda)$ as a module for the  Heisenberg-Virasoro vertex algebra at level zero \cite{AR-2017}. \footnote[1]{  It was proved in \cite{AR-2017} that  $\Pi _{(-1)}  (\lambda)$  is a direct sum of irreducible modules for the Heisenberg-Virasoro vertex algebra at level zero, and $e^c_0$ is a homomorphism which acts non-trivially on each irreducible component.}

 \begin{lemma}  \label{injektivnost}
 The operator $e^c _0$ acts injectively  on $\Pi _{(-1)}  (\lambda)$.
 \end{lemma}

Let $k \in {\C}$.  Now we shall fix the Heisenberg and the Virasoro vector in $\Pi(0)$, and calculate the character of the weight  $\Pi(0)$--modules.

Vector
\bea  &&  \omega_{\Pi(0)}  =  \frac{1}{2} c(-1) d(-1)  -\frac{1}{2} d(-2) + \frac{k}{4 } c(-2) \label{virasoro-pi0} \eea is a  Virasoro vector in the vertex algebra ${\Pi(0)}$ of central charge $ \overline c= 6 k +2$.
The Virasoro field is $\overline L(z) = \sum_{n\in {\Z}} \overline L(n) z ^{-n-2}$.

Define the Heisenberg vector   $h = 2 \mu = \frac{k}{2} c + d$ and the corresponding field $h(z)\hspace{-1pt} =\hspace{-1pt} \sum_{n \in {\Z}} h(n) z^{-n-2}$. Then
$ [h(n), h(m) ] = 2 k \delta_{n+m,0}$.

\begin{definition} A module $M$ for the vertex algebra $\Pi(0)$ (resp. $\Pi^{1/2} (0)$ )   is called weight if  the operators  $\overline L(0)$ and $h(0)$ act semisimply  on $M$.
\end{definition}
Assume that $M$ is a weight module for the vertex algebra $\Pi(0)$ (resp. $\Pi^{1/2} (0)$) having finite-dimensional weight spaces for $(\overline L(0), h(0))$. Then we can define the character of $M$:
$$ \mbox{ch}[M]  (q,z) = \mbox{Tr}_{M }   q^{\overline L(0) - \overline c /24}  z ^{h(0)}. $$
\begin{proposition} \label{char-pi0}
\item[(1)]   For every $\lambda \in {\C} $, $\Pi _{(-1)}  (\lambda) $ is a  $\Z_{\ge 0}$--graded weight $\Pi(0)$--module with character
$$ \mbox{ch} [\Pi _{(-1) }  (\lambda) ](q,z)=  \frac{  z^{-k + 2 \lambda } \delta (z^2)  }{\eta(\tau) ^2 }.$$
\item[(2)]  For every $\lambda \in {\C} $, $\Pi ^{1/2} _{(-1)}  (\lambda) $ is a  $\Z_{\ge 0}$--graded weight $\Pi^{1/2} (0)$--module with character
$$ \mbox{ch} [\Pi ^{1/2} _{(-1) }  (\lambda) ](q,z)=  \frac{  z^{-k + 2 \lambda } \delta (z)  }{\eta(\tau) ^2 }.$$

where  $\delta(z) = \sum_{\ell  \in {\Z} } z ^{ \ell}$, $\eta(\tau) = q^{1/24} \prod_{n \ge 1} (1-q^{n})$.
\end{proposition}
\begin{proof}  Set  $M=\Pi _{(-1)}  (\lambda)$. First we notice that $M$ is a $\Z_{\ge 0}$--graded   $M= \bigoplus_{m =0} ^ {\infty} M(m)$ such that
\bea  && M(0) \cong \mbox{span}_{\C} \{ e^{-\mu + ( \lambda + j) c} \ \vert \ j \in {\Z}\}, \nonumber \\
&&  \overline L(0) \vert M(0) \equiv \frac{k}{4} \mbox{Id}, \quad h(0) e^{-\mu + ( \lambda + j) c}  = ( - k + 2 (\lambda + j) )e^{-\mu + ( \lambda + j) c} . \nonumber
\eea
Now we have
 \bea
\mbox{ch}[  \Pi _{(-1)}  (\lambda) ]  (q,z) &=& \mbox{Tr} _{  \Pi _{(-1)}  (\lambda) }  q^{ \overline L (0) -\overline c /24 } z^{ h(0)}  \nonumber \\
&=&      q^{-\frac{k}{4} - \frac{1}{12} }  q^{\frac{k}{4} }  \frac{ z^ { -k + 2 \lambda}  \delta(z^2)} {\prod_{n =1} ^{\infty}  (1-q^n) ^2  }
= \frac{  z ^{- k + 2 \lambda} \delta (z ^2) }{\eta (\tau) ^2}. \nonumber
 \eea
This proves the assertion (1). The proof of (2) is analogous.
\end{proof}

 \subsection{Whittaker $\Pi(0)$--modules}

The construction of Whittaker modules for the   vertex algebra $\Pi(0)$--modules were presented in  \cite[Section 11]{ALZ-2016}. We considered   $\mathcal A$--module
$U_{\lambda}$ generated by vector $v_1$ such that
$$ e^{nc} v_{1} = {\lambda} ^n v_1 \quad (n \in {\Z}).$$
Note that  as a vector space $U_{\lambda} \cong {\C}[d]$  with the  free action of $d$.
Then we proved that  $\Pi_{\lambda} = U_{\lambda} \otimes M(1)$ has that structure of an  irreducible Whittaker $\Pi(0)$--module with the Whittaker vector $w_{\lambda} = 1 \otimes v_1$.

Similarly we can construct Whittaker modules for the vertex algebra   $\Pi ^{1/2} (0)$.  Consider $\mathcal A^{1/2}$--module $U_{\lambda} ^{1/2} ={\C}[d(0)] v _2$ such that
$$ e^{nc} v_{2} = {\lambda} ^n v_2 \quad (n \in  \tfrac{1}{2}{\Z}).$$ Then $ U_{\lambda} ^{1/2} \otimes M(1) $ has the structure of an irreducible  $g$--twisted  $\Pi ^{1/2} (0)$--module.

  \begin{proposition} \label{constr-pi0-mod-whit}

 \item[(1)]  \cite[Theorem 11.1] {ALZ-2016}  For every  $\lambda \in {\C} \setminus \{0 \}$ there is an irreducible $\Pi(0) $--module $\Pi _{\lambda}$ so that
 $c(0)$ acts on  $\Pi _{\lambda}$  as $- \mbox{Id}$ and that $\Pi _{\lambda}$ is generated by cyclic vector $w_{\lambda} $ satisfying
 $$ e^{c}_{0} w_{\lambda} = {\lambda} w_{\lambda}, \quad e^{-c} _0 w_{\lambda} = \frac{1}{\lambda} w_{\lambda}. $$
 As a vector space, $\Pi_{\lambda} = M(1) \otimes {\C}[d(0)]$.

  \item[(2)] For every  $\lambda \in {\C} \setminus \{0 \}$  $\Pi_{\lambda}$ has the structure of an irreducible $g$--twisted  $\Pi ^{1/2} (0)$--module generated by cyclic vector $w_{\lambda}$  such that
  $$ e^{c/2}_{0} w_{\lambda} = \sqrt{\lambda} w_{\lambda}, \quad e^{-c/2} _0 w_{\lambda} = \frac{1}{\sqrt{\lambda}} w_{\lambda}. $$
 \end{proposition}

 \begin{remark}
 Note that in \cite[Theorem 11.1] {ALZ-2016}, the operator $e^c_0$ is denoted by $a(0)$.
 \end{remark}

 \begin{remark}
 The operator $\bar L(0)$ acts semi-simply on $\Pi _{\lambda}$. But the action of $h(0)$ is not diagonalizable. This can be easily seen from  the action of $h(0)$ on top component $\Pi_{\lambda}(0) \cong U_{\lambda}$:
 $$ h(0)  \equiv \frac{k}{2} c(0) + d(0)  = - \frac{k}{2} + d(0)$$
 which is not  diagonalizable.
 \end{remark}

 \section{Realization of the admissible affine vertex algebra $V_k(sl(2))$ }
\label{realization-admissible}

In this section we use the realization from Proposition  \ref{prop-sl2-1}  and get a realization of the admissible affine vertex algebra $V_k(sl(2))$.

 Assume now that $k$ is admissible and $k \notin{\Z}$. Then
 $$k +2 = \frac{p'}{p}, \quad d_k = 1- 6 \frac{(p-p') ^2}{p p'} = d_{p,p'}. $$
 Let $L^{Vir} (d_{p,p'},0)$ be the simple rational vertex operator algebra of central charge $d_{p,p'}$ (cf. Subsection \ref{Vir-subsect}).

Let now $\varphi = p' \gamma $. Since  $ \langle \varphi, \varphi \rangle = \frac{2 p'   {^2}  } {k+2}     =   2 p p' $, we
set  $M_{\varphi} (2 p p', 0) = M_{\gamma} (\frac{2}{k+2}, 0)  $ and
$$\omega ^{(k)}  =\left(\frac{1}{4 p p'} \varphi (-1) ^2 + \frac{p-p'}{ 2 p p'} \varphi (-2)\right) {\bf 1} .$$

 The universal vertex algebra $V ^{Vir} (d_{p,p'}, 0)$ is not simple and it contains a non-trivial ideal generated by singular vector
 $ \Omega_{p,p'} ^{Vir} $  of conformal weight  $(p-1) (p'-1)$.
 Moveover,
 $$ L ^{Vir} (d_{p,p'}, 0) = \frac{V^{Vir} (d_{p,p'}, 0) } {  U(Vir). \Omega_{p,p'} ^{Vir} }  $$
 is a simple vertex algebra (minimal model).
  The singular vector  $\Omega_{p,p'} ^{Vir}$ can be constructed in the free--field realization using screening operators.

  \begin{proposition} \label{prop-vir} \cite{TK}, \cite{RW1}
  There exist a unique, up to a scalar factor,   $Vir$--homomorphism
 \bea
 \Phi_{p,p'} ^{Vir}:   M_{\varphi} ( 2 p p', 0) . e ^{-\tfrac{p'-1}{p'} {\varphi}} &\rightarrow&   M_{\varphi} ( 2 p p', 0) \nonumber \\
  e^{-\frac{p'-1}{p'} \varphi } &\mapsto &  \Omega_{p,p'} ^{Vir} \nonumber .
 \eea
  There is a cycle $\Delta_{p'-1}$ and a non-trivial scalar $c_{p-1}$  such that  $ \Phi_{p,p'} ^{Vir} $ can be be represented as
$$ \frac{1}{c_{p'-1}} \int_{\Delta_{p'-1}}  Y( e^{\frac{\varphi} {p'} }  , z_1) \cdots   Y( e^{\frac{\varphi}{p'} }, z_{p'-1} ) dz_1 \cdots d z_{p'-1}. $$
  \end{proposition}
Then $ \omega_{p,p'} = \omega^{(k)} + U(Vir). \Omega_{p,p'} ^{Vir} $ is the conformal vector  in  $L ^{Vir} (d_{p,p'}, 0) $.

 The universal affine vertex algebra $V^ k(sl(2))$ also contains a non-trivial maximal ideal generated by $\widehat{sl(2)}$--singular vector
 $\Omega_k ^{sl(2)}$   of conformal weight  $p (p'-1)$.
 Moreover,
 $$V_k(sl(2)) = \frac{V^ k(sl_2)}{ U (\widehat{sl(2)}). \Omega_k ^{sl(2)}}$$
 is a simple, admissible vertex algebra.  Let $\omega_{sug}$ denotes the Sugawara Virasoro vector in $V_k (sl(2))$ od central charge $\frac{3 k}{ k+2}$.
 The singular vector $ \Omega_k ^{sl(2)}$ can be also constructed using screening operators. The proof  was presented in \cite[Theorem 3.1] {RW2}   for $\widehat{sl(2)}$ and in \cite[Proposition 6.14]{Ara1}  in a more general setting (see also \cite{Ara2}  for some applications).
  \begin{proposition} \cite{RW2}, \cite{Ara1}
   There exist a unique, up to a scalar factor,   $\widehat{sl(2)}$--homomorphism
 \bea
 \Phi_{k} ^{sl(2)}:   W_{k, 2 (p'-1) } &\rightarrow&   W_{k,0}\nonumber \\
  e^{-\frac{p'-1}{p'} \varphi  + (p'-1) (\alpha +\beta)  } &\mapsto &  \Omega_{k} ^{sl_2} \nonumber .
 \eea

 \end{proposition}

By \cite[Theorem 3.1] {RW2}  $   \Phi_{k} ^{sl(2)}$ can be   represent as
$$ \frac{1}{c_{p'-1}} \int_{\Delta_{p'-1}}   U ( z_1 )  \cdots   U( z_{p'-1} ) dz_1 \cdots d z_{p'-1}, $$
where $U(z) = Y (a(-1) e^{-\frac{\delta}{k+2} }, z) $ and the cycle $\Delta_{p'-1}$  is as in  Proposition   \ref{prop-vir}.
But  since $U(z) = Y(e^{\frac{\varphi}{p'}}, z)$ we get the following consequence:
   \begin{corollary} \label{posljedica}  $\Phi_{k} ^{sl(2) }$ can be extended to a $\widehat{sl(2)}$--homomorphism
   $$ M_{\varphi} ( 2 p p', 0) . e ^{-\tfrac{p'-1}{p'} {\varphi}} \otimes \Pi (0) \rightarrow  M_{\varphi} ( 2 p p', 0)  \otimes \Pi (0) $$ such that
   $ \Phi_{k} ^{sl(2)} = \Phi_{p,p'} ^{Vir} \otimes \mbox{Id} $ and $ \Omega_{k} ^{sl_2} =  \Omega_{p,p'} ^{Vir} \otimes e^{ (p'-1) c }. $
   \end{corollary}

 \begin{example}
 Let us illustrate the above analysis in the simplest case $p' = 2$. Then we have that $k+ 2 = \frac{2}{p}$ where $p$ is odd natural number, $p \ge 3$.   Moreover, we have
 $\langle \varphi, \varphi \rangle = 4p$. The construction of the Virasoro singular vectors was obtained in \cite{AdM-IMRN} by using lattice vertex algebras.

 The singular vector in $V^k (sl(2))$ is given by
 \bea
 Q e^{\frac{\delta}{k+2}}  &=& (a (-1) e^{-\frac{\delta}{k+2}})_0  e^{\frac{\delta}{k+2}}  \nonumber \\
 &=&  S_{p-1} ( \alpha + \beta -   \frac{\delta}{ k+2} ) a(-1) {\bf 1} \nonumber  \nonumber \\
 &=&  S_{p-1} ( \frac{\varphi}{2} )  e^{\alpha + \beta}  \nonumber \\
 &=& Q e^{-\frac{\varphi} {2} } \otimes e^{ \alpha + \beta}  = \Omega_{p,2 } ^{Vir} \otimes e^{  c }. \nonumber
 \eea
 Here $S_{n} (\gamma)$ denotes  the $n$-th  Schur polynomial in $(\gamma (-1), \gamma(-2) , \dots)$.
In particular,   $Q e^{-\frac{\varphi} {2}  } = S_{p-1} ( \frac{\varphi}{2}  ) $ is a singular vector in $V^{Vir} (d_{p,2}, 0)  \subset M_{\varphi} (4p)$ (cf. \cite{AdM-IMRN}).

 \end{example}

Finally we get the realization of $V_k (sl(2))$:
 \begin{theorem} \label{real-simpl-sl2}
 There exist a non-trivial $\widehat{sl(2)}$--homomorphism
 $$ \overline{\Phi} : V_k (sl(2)) \rightarrow L ^{Vir} (d_{p,p'}, 0) \otimes \Pi(0) $$
 defined by the relations (\ref{def-e-3})-(\ref{def-f-3}).
 Moreover,
 \bea  \overline{\Phi} (\omega_{sug}) &=& \omega_{p,p'} +  \frac{1}{k} \mu (-1) ^2 - \frac{1}{k} \nu (-1) ^2 - \nu(-2)  \label{sug} \\
 &=& \omega_{p,p'} + \frac{1}{2} c(-1) d(-1)  -\frac{1}{2} d(-2) + \frac{k}{4 } c(-2). \label{sug-cd}
 \eea
 \end{theorem}
 \begin{proof}
 We have constructed homomorphism
$ \Phi : V^ k (sl(2)) \rightarrow V ^{Vir} (d_{p,p'}, 0) \otimes \Pi(0) $
and showed in Corollary \ref{posljedica}     that
$ \Phi( \Omega_k ^{sl(2)}) = \Omega_{p,p'} ^{Vir} \otimes e ^{ (p'-1) c }$.
The claim  follows.
 \end{proof}

 In what follows, we identify $ \omega_{sug}$ with  $\overline{\Phi} (\omega_{sug}) $ and denote  the Sugawara Virasoro field by
 $$L_{sug} (z ) = \sum_{n \in {\Z}} L_{sug} (n) z^{-n-2}, \quad L_{sug} (n) = (\omega_{sug})_{n+1}. $$

\begin{remark}  \label{central-charge} Note that
$  \overline{\Phi} (\omega_{sug})  = \omega_{p,p'} + \omega_{\Pi(0)} $,  where $  \omega_{\Pi(0)}$
 is a Virasoro vector in the vertex algebra ${\Pi(0)}$ of central charge $6 k +2$ given by   (\ref{virasoro-pi0}).
In particular, we have
$$c_{sug} = \frac{3 k}{k+2} = d_{p,p'} + 6 k  +2. $$
\end{remark}

 \section{Realization  of ordinary $V_k(sl(2))$--modules and their intertwining operators}

 In this section we present a realization of irreducible, ordinary $V_k(sl(2))$--modules, i.e.,   the $V_k(\g)$--modules having finite-dimensional $L_{sug}(0)$--eigenspaces.
 It was proved in \cite{CHY}, that the category of ordinary $V_{k}(sl(2))$--modules at the admissible level , denoted by $\mathcal O_{k,ord}$, form a braided tensor category with the tensor product bifunctor $\boxtimes_{\mathcal P(1)}$.

  In this section, we show that the  intertwining operators among ordinary $V_k(\g)$--modules can be constructed from  the intertwining operators for the minimal Virasoro vertex algebra.

\vskip 5mm Recall \cite{AdM-MRL} that the set
$$ \{  \mathcal L _s := L_{A_1} ( (k+ 1 -s )\Lambda_0 + (s-1)  \Lambda_1)  \ \vert \ s = 1, \dots, p'-1 \}$$
provides all irreducible, ordinary $V_k(\g)$--modules.

\begin{proposition}  \label{ordinary}
Let $s \in {\Z}$, $ 1 \le s \le p'-1$.
We have
$$  \mathcal L _s \cong  V_k(sl(2)) . (v_{1,s}  \otimes e^{\frac{s-1}{2} c} ) \subset  L^{Vir} (d_{p,p'}, h_{p,p'} ^{1,s} ) \otimes \Pi ^{1/2} (0). $$
\end{proposition}
 \begin{proof}
 Since $L^{Vir} (d_{p,p'}, h_{p,p'} ^{1,s} ) \otimes \Pi ^{1/2} (0)$ is a $V_k(sl(2))$--module, it suffices to show that  $w_{s} =v_{1,s}  \otimes e^{\frac{s-1}{2} c}$ is a singular vector for $\widehat{sl(2)}$.  For $n \ge 0$ we  have:
 \bea
 && e(n) w_{s}  =   f(n +2) w_s = 0,  \
  h(n) w_{s} =  (s-1) \delta_{n,0} w_{s}. \nonumber
 \eea
It remains to prove that $f(1) w_s = 0$. Since
 \bea
(w_s) _1  f  &= & (k+2) (v_{1,s} )_1 \omega_{p,p'} \otimes e ^{ \frac{s-1}{2} c - c} - v_{1,s} \otimes   e ^{ \frac{s-1}{2} c} _1 ( \nu(-1) ^2 + (k+1) \nu(-2)  ) e^{-c}  \nonumber  \\
& = &  \left(  (k+2)      h_{p,p'} ^{1,s}  -  \frac{ (s-1) ^2}{4}   +  \frac{ (k+1)  (s-1)}{2}     \right) v_{1,s}    \otimes   e^{ \frac{s-1}{2} c - c} = 0,
 \eea
 we get
 $ f(1) w_s = - [(w_s )_{-1}, f (1)] {\bf 1} = (w_s) _1 f   = 0$.
 The proof follows.
 \end{proof}

\vskip 5mm

 The following fusion rules result was proved in \cite{CHY}:
 \bea  && \mathcal L _{s_1} \boxtimes_{\mathcal P(1)}   \mathcal   L _{s_2} = \bigoplus _{s_3 = 1} ^ {p'  -1} N_{s_1, s_2} ^{s_3}  \mathcal  L_{s_3}, \label{affine-fusion} \eea
 where the fusion coefficient is
 $$N_{s_1,  s_2 } ^{s_3}:=
  \begin{cases} 1 & \text{if}  \ \vert s_2 - s_1 \vert + 1 \le s_3 \le \mbox{min} \{ s_1 + s_2 +1, 2 p'- s_1 - s_2 - 1\} \\  & \ \ s_1 + s_2 + s_3 \ \mbox{odd} \\  0 & \text{otherwise} \end{cases}.$$
Moreover,  $ N_{s_1,  s_2 } ^{s_3}=\left[\begin{matrix} (1, s_3) \\ (1, s_1 ) \ (1, s_2)\end{matrix}\right]_{ (p,p') } $
 coincides with the  fusion coefficient for the Virasoro minimal models (cf. Subsection \ref{Vir-subsect}):
 $$  L^{Vir} (d_{p,p'}, h_{p,p'} ^{1,s_1} ) \boxtimes_{\mathcal P(1)}   L^{Vir} (d_{p,p'}, h_{p,p'} ^{1,s_2} )  =  \bigoplus _{s_3 = 1} ^ {p'  -1} N_{s_1, s_2} ^{s_3} L^{Vir} (d_{p,p'}, h_{p,p'} ^{1,s_3} ).$$

 By using our realization we can  construct all  intertwining operators appearing in the fusion rules  (\ref{affine-fusion}) as follows.
  Let $\mathcal Y_1(\cdot, z)$ be a non-trivial intertwining operator of the type $$\binom{   L^{Vir} (d_{p,p'}, h_{p,p'} ^{1,s_3})}{L^{Vir} (d_{p,p'}, h_{p,p'} ^{1,s_1})  \ \ L^{Vir} (d_{p,p'}, h_{p,p'} ^{1,s_2}) } $$
 for the Virasoro vertex operator algebra $L^{Vir} (d_{p,p'}, 0)$. We can tensor it with the vertex operator map  $Y_{ \Pi ^{1/2} (0)} (\cdot, z)$ for the vertex algebra $ \Pi ^{1/2} (0)$, and obtain  the intertwining operator $\mathcal Y = \mathcal Y_1 \otimes Y_{ \Pi ^{1/2} (0)}$ of type
 \bea  \binom{   L^{Vir} (d_{p,p'}, h_{p,p'} ^{1,s_3})  \otimes    \Pi ^{1/2} (0)}{L^{Vir} (d_{p,p'}, h_{p,p'} ^{1,s_1})  \otimes \Pi ^{1/2} (0)   \ \ L^{Vir} (d_{p,p'}, h_{p,p'} ^{1,s_2})  \otimes \Pi ^{1/2} (0) }\label{operator-restrikcija-2}\eea
in the category of  $V_k(sl(2))$--modules.
Intertwining operator corresponding to the fusion rules  (\ref{affine-fusion}) can be obtained by restricting the above intertwining operators.

 \begin{proposition}
 Assume that $N_{s_1, s_2} ^{s_3} = 1$. Then there is a non-trivial intertwining operator of  type
  $$\binom{  \mathcal  L_{s_3}}{\mathcal  L_{s_1} \ \  \mathcal  L_{s_2} }, $$
realizead as a restriction of the intertwining operator (\ref{operator-restrikcija-2}).
\end{proposition}
 \begin{proof}
Note that $ \mathcal L_{s_i} = V_k(sl(2)). w_{s_i}$, where $w_{s_i}=v_{1,s_i} \otimes e^{\frac{s_i -1}{2} c}$, $i=1,2,3$. By restricting $\mathcal Y(\cdot, z)$ on $\mathcal L_{s_1} \otimes  \mathcal L_{s_2}$, we get a non-trivial  intertwining operator  of type
  $\binom{  \mathcal M_{s_3}}{\mathcal  L_{s_1} \ \  \mathcal  L_{s_2} }, $
   where $M_{s_3} = V_{k} (sl(2)). v_{s_3} \otimes e^{ \frac{s_1 + s_2-2}{2} c}$.
   Note that $ s_1 + s_2 - s_3 -1 \in  2 {\Z}_{\ge 0}$. Then
   $$ e(-1) ^{\frac{s_1 + s_2 - s_3-1}{2} }   w_{s_3} = e ^{\frac{s_1 + s_2 - s_3-1}{2} c }_{-1}  \left( v_{s_3} \otimes e^{ \frac{s_3-1}{2} c} \right)  = v_{s_3} \otimes e^{ \frac{s_1 + s_2-2}{2} c}.$$
    This shows that $M_{s_3} \subseteq \mathcal  L_{s_3}$, and since $ \mathcal  L_{s_3}$ is irreducible, we have that $\mathcal  L_{s_3} =M_{s_3}$.
Thus, we have  constructed a non-trivial intertwining operator of type  $\binom{  \mathcal  L_{s_3}}{\mathcal  L_{s_1} \ \  \mathcal  L_{s_2} }$.
The proof follows.
 \end{proof}

\section{Explicit realization  of  relaxed highest weight $V_k(sl(2))$--modules}

 \label{relaxed-realization}

 We say that a ${\Z}_{\ge 0}$--graded $V^ k(sl(2))$--module  $M= \bigoplus_{m =0} ^{\infty} M(m)$ is a relaxed highest weight module if the following conditions hold:
 \begin{itemize}
 \item Each  graded component $M(m)$  is an eigespace for $L_{sug}(0)$;
 \item $M= V^{k} (sl(2)). M(0)$;
 \item $M(0)$ is an  irreducible weight $sl(2)$--module which is neither highest nor  lowest weight $sl(2)$--module.
 \end{itemize}
 The subspace $M(0)$ is usually called the top component of $M(0)$ (although it has lowest conformal weight with respect to $L_{sug}(0)$).

 By  using the  classification of irreducible   $V_k(sl(2))$--modules from \cite{AdM-MRL} (see also \cite{RW2}),  we conclude that   any  irreducible  $\Z_{\ge 0}$--graded $V_k(sl(2))$--module belongs to one of the following series:
 \begin{itemize}
 \item[(1)]   The ordinary modules   $\mathcal L_s$   (cf. Proposition \ref{ordinary}) with  lowest conformal  weight
 $ h_{p,p'} ^{1,s} +  \frac{ s-1}{2}$.
 \item[(2)]  The  ${\Z}_{\ge 0}$ graded $V_{k}(sl(2))$--modules $D_{r,s} ^{\pm}$, $1 \le r \le p-1$, $1 \le s\le p'-1$, where
\item  $D_{r,s} ^+$ is an  irreducible $\Z_{\ge 0}$--graded $V_k(sl(2))$--module such that  $D_{r,s} ^+ (0)$  is an  irreducible highest weight $sl(2)$--module with highest weight $\mu_{r,s} =  (s-1 - (k+2) r) \omega_1$, where $\omega_1$ is the fundamental weight for $sl(2)$.
\item $D_{r,s} ^-$  is an  irreducible $\Z_{\ge 0}$--graded $V_k(sl(2))$--module such  that   $D_{r,s} ^- (0)$  is an  irreducible  lowest  weight $sl(2)$--module with lowest weight $- \mu_{r,s}$.
 \item[(3)] Relaxed   highest weight modules $M = \bigoplus_{m = 0} ^ {\infty}  M(m)$, such that the top component $M(0)$ has conformal weight   $h_{p,p'} ^{r,s} + k /4$.
   \end{itemize}
 In this section we construct a family of relaxed highest weight modules for $V_k(sl(2))$. These  modules  also appeared in \cite{AFR}, \cite{E}, \cite{FST},  \cite{CR-2013}, \cite{RW2}, \cite{RSW}, \cite{S-2017}. In this section  we shall explicitly construct these modules and see from  the realization that their  characters are given   by  the Creutzig-Ridout character formulas \cite{CR-2013} (see also \cite{KR}).
 \subsection{Realization of relaxed $V_k(sl(2))$--modules}

  For every $\lambda \in {\C}$ and $r, s \in \Z$, $ 0 < r < p$, $0 < s < p'$ we define  the $L^{Vir}(d_{p,p'},0 ) \otimes \Pi (0)$--module
$$
 \mathcal{E}_{r,s} ^{\lambda} =
L^{Vir}(d_{p,p'},h_{p,p'} ^{r,s}) \otimes \Pi_{-1} (\lambda)  =L^{Vir} (d_{p,p'}, h_{p,p'} ^{r,s}) \otimes \Pi(0). e^{- \mu + \lambda \frac{2}{k} (\mu - \nu)}.$$

  Let $\ell  \in {\Z}$ and $\pi_{\ell}$ be the (spectral flow) automorphism of $V_k(sl(2))$  defined by
 $$ \pi_{\ell}  (e(n) ) = e(n +  \ell), \quad \pi_{\ell}  (f(n) ) = f (n - \ell), \quad \pi_{\ell}  (h(n)) = h(n)  + \ell k \delta_{n,0}. $$
 By using the realization of $V_k(sl(2))$ one can  see that the spectral-flow automorphism $\pi_{\ell}$ can be realized  as the lattice element $e ^{\ell \mu}$ acting on $\Pi(0)$--modules:

\begin{proposition}  \label{spectral-flow} We have: $ \pi_{\ell}  ( \mathcal{E}_{r,s} ^{\lambda} ) = L^{Vir}(d_{p,p'},h_{p,p'} ^{r,s}) \otimes \Pi_{ \ell -1} (\lambda). $ \end{proposition}
\begin{proof}
Using \cite[Proposition 2.1]{A-2007} we get that if $(M, Y_M(\cdot, z) )$ is a $V_k(sl(2))$--module then
$$(\pi_{\ell} (M), Y_{\pi_{\ell} (M)) } (\cdot, z) ) :=  (M, Y_M ( \Delta (\tfrac{\ell h}{2},z) \cdot, z)  )$$
Using Lemma \ref{Li} we get
$$ (\Pi_{\ell -1} (\lambda) , Y_{ \Pi_{\ell-1} } (\cdot, z) ) =   ( \Pi_{-1} (\lambda)  , Y_{ \Pi_{-1} (\lambda)  } ( \Delta (\ell \mu ,z) \cdot, z)  ) $$
which implies  $\pi_{\ell} ( \mathcal{E}_{r,s} ^{\lambda} ) = L^{Vir}(d_{p,p'},h_{p,p'} ^{r,s}) \otimes \Pi_{\ell -1} (\lambda)$. The proof follows.
\end{proof}

 Let
 $$ E_{r,s} ^{\lambda} = v_{r,s} \otimes e^{- \mu + \lambda \frac{2}{k} (\mu - \nu)} =v_{r,s} \otimes  e ^{\beta - \delta /2 + \lambda (\alpha + \beta)}. $$
 Then $E_{r,s} ^{\lambda} $ is a primary vector with conformal weight $ k /4 + h_{p,p'} ^{r,s}$, i.e.
 \bea L_{sug} (n) E_{r,s} ^{\lambda} =  (k/ 4 + h_{p,p'} ^{r,s}) \delta_{n,0}  E_{r,s} ^{\lambda}    , \quad (n \ge 0). \label{act-L} \eea
 The $sl(2)$ action on these vectors is as follows:
 \bea
 e(0) E_{r,s} ^{\lambda}  &=& E_{r,s} ^{\lambda +1}, \label{act-e} \\
 h(0)  E_{r,s} ^{\lambda} &=& (-k + 2 \lambda )E_{r,s} ^{\lambda}, \label{act-h} \\
 f(0) E_{r,s} ^{\lambda} &=& \left( (k+2)  h_{p,p'} ^{r,s} - \lambda ^2 + \lambda (k+1) ) \right)   E_{r,s} ^{\lambda-1}  \nonumber \\
  &=& \left( \frac{ (s p - r p' ) ^2 }{4 p ^2} - (\lambda -\frac{p'- p}{2 p} ) ^2 \right)   E_{r,s} ^{\lambda-1} . \label{act-f}
 \eea

 \begin{remark} \label{indecomposable-1}
 Note that
 $ f(0)  E_{r,s} ^{\lambda} = 0 $ iff $ \lambda = {\lambda}_{r,s} ^{\pm}  $ where $\lambda_{r,s} ^{\pm} = \frac{p'-p}{2p} \pm \frac{s p -r  p'}{2p}. $
It is also important to notice that
$ \lambda_{r,s} ^{+} = \frac{ s- 1}{2} - \frac{r-1}{2}  (k+2)$, $\lambda_{r,s} ^{-} = \lambda_{ p-r, p'-s} ^+$.
 %

If  $\lambda = \lambda_{r,s} ^{+}   $,  then
$  \mathcal{E}_{r,s} ^{\lambda } $ is an indecomposable $\Z_{\ge 0}$--graded $V_k(sl(2))$--module  which appears in the non-split extension
\bea \label{non-split} 0 \rightarrow  D _{p-r,p'-s}   ^-    \rightarrow \mathcal{E}_{r,s} ^{\lambda } \rightarrow  D_{r,s} ^+  \rightarrow 0. \eea
where $D_{r,s} ^{\pm}$ are irreducible $V_k(sl(2))$--modules described above.
 This extension was also constructed in  \cite{CR-2013} (see \cite[Section 4]{CR-2013} and their formula (4.3)).
In Section \ref{logarithmic}, we shall see that indecomposable modules $  \mathcal{E}_{r,s} ^{\lambda } $  appear in the construction of logarithmic modules.
 \end{remark}

 Assume that $\lambda \notin  \lambda_{r,s} ^{\pm} + {\Z} $.

    \begin{theorem}
 We have:
 \item[(1)] $ \mathcal{E}_{r,s} ^{\lambda}$ is $\Z_{\ge 0}$--graded $V_k(sl(2))$--module.
 \item[(2)] The  top component  is
 $ \mathcal{E}_{r,s} ^{\lambda} (0) =\mbox{span}_{\C} \{ E_{r,s} ^{\lambda + j}  \ \vert \ j \in {\Z} \} $
and it  has conformal weight $k/4 + h_{p,p'} ^{r,s}$. If $\lambda \notin  (\lambda_{r,s} ^{\pm} + {\Z}) $, then  $ \mathcal{E}_{r,s} ^{\lambda} (0)$ is an irreducible $sl(2)$--module.
 \item[(3)]  The character of $ \mathcal{E}_{r,s} ^{\lambda}$ is given by  $$ \mbox{ch}[\mathcal{E}_{r,s} ^{\lambda} ]  (q,z) =  z ^{- k + 2 \lambda}   \chi_{r,s} (q)   \frac{ \delta (z ^2)} {\eta (\tau) ^2} , $$
     where $\chi_{r,s}(q) = \mbox{ch}[ L^{Vir} (d_{p,p'}, h_{p,p'} ^{r,s}  ) ]  (q) $, $\delta(z^2) = \sum_{\ell  \in {\Z} } z ^{2 \ell}$, $\eta(\tau) = q^{1/24} \prod_{n \ge 1} (1-q^{n})$.
 \end{theorem}
 \begin{proof}
 By using (\ref{act-L}) we see that $ \mathcal{E}_{r,s} ^{\lambda}$ is  ${\Z}_{\ge 0}$--graded, and by (\ref{act-e})-(\ref{act-f}) we get that the top component   $ \mathcal{E}_{r,s} ^{\lambda} (0)$  is an  irreducible $sl(2)$--module. This proves assertions (1) and (2).  Recall that  $ c_{sug} = \frac{3k}{k+2} = d_{p,p'} +  2  + 6 k$ (see Remark \ref{central-charge}). Using Proposition \ref{char-pi0} our explicit realization gives the following character formula:
 \bea
\mbox{ch}[\mathcal{E}_{r,s} ^{\lambda} ]  (q,z) &=& \mbox{Tr} _{ \mathcal{E}_{r,s} ^{\lambda}  }  q^{ L_{sug} (0) -c_{sug} /24 } z^{ h(0)}  \nonumber \\
&=&   \mbox{Tr} _{ L^{Vir} (d_{p,p'}, h_{p,p'} ^{r,s}  ) }  q^{ L  (0) -d_{p,p'} /24 }   \mbox{ch} _{ \Pi_{(-1)} (\lambda)} (q,z) \nonumber  \\
& =&    \mbox{ch}[ L^{Vir} (d_{p,p'}, h_{p,p'} ^{r,s}  ) ]  (q) \cdot \frac{  z ^{- k + 2 \lambda} \delta (z ^2) }{\eta (\tau) ^2}. \nonumber
 \eea
 The proof follows.
 \end{proof}


 \subsection{Irreducibility of relaxed $V_k(sl(2))$--modules}
We will now discuss the irreducibility of relaxed  $V_k(sl(2))$--modules $\mathcal{E}_{r,s} ^{\lambda}$. We shall present a proof of irreducibility in generic cases (cf. Proposition  \ref{maximal}) which uses our realization and the representation theory of the vertex operator algebra $V_k(sl(2))$.
 \begin{lemma} \label{beskonacne}
 Assume that $M=\bigoplus_{n \in {\Z}_{\ge 0} } M(n)$ is an irreducible  $\Z_{\ge 0}$--graded $V_k(sl(2))$--module such that   $M(0)$ is  an irreducible, infinite-dimensional weight $sl(2)$--module.  Then $M$  is isomorphic to a subquotient of  $ \mathcal{E}_{r,s} ^{\lambda}$ for  certain $1 \le r \le p'-1$, $1 \le s \le p-1$,  $\lambda \in {\C}$ and
 $$ L_{sug} (0) \equiv (h_{p,p'} ^{r,s} + k/4) \mbox{Id} \quad  \mbox{on} \ M(0). $$
 \end{lemma}
 \begin{proof}
 If  $M(0)$ is an irreducible highest (resp. lowest)  weight module for $sl(2)$, then the classification of irreducible $V_k(sl(2))$--modules gives that $M \cong D_{r,s} ^+$ (resp. $M \cong D_{p-r,p'-s} ^-.$).  By Remark   \ref{indecomposable-1}, $M$ can be realized as a submodule or a quotient of the indecomposable module $\mathcal E_{r,s} ^{\lambda}$. If $M$ is an irreducible relaxed $V_k(sl(2))$--module, then $M(0) \cong \mathcal E_{r,s} ^{\lambda} (0)$ for certain $r,s, \lambda$, and therefore $M$ is isomorphic to a (quotient) of  $E_{r,s} ^{\lambda}$. The proof follows.
 \end{proof}

\begin{remark} \label{ired-comment} Modules $ \mathcal{E}_{r,s} ^{\lambda}$    are irreducible for   $\lambda \notin \lambda_{r,s} ^{\pm} + {\Z} $. This follows from the fact that they have the same characters as irreducible quotients of relaxed Verma modules presented by T. Creutzig and D. Ridout in  \cite{CR-2013}.  We should mention that a new proof of irreducibility of a large family of relaxed highest weight modules  is presented in new paper   \cite{KR} using Mathieu' s coherent families \footnote[2]{Talk presented by K. Kawasetsu at the conference \emph{Affine, vertex and W-algebras, Rome, December 11-15, 2017}}.
 K. Sato in \cite{S-2017} presented a proof of irreducibility of certain typical modules for the $N=2$ superconformal algebra   which are related to  the relaxed $\widehat{sl(2)}$--modules via the  anti Kazama--Suzuki mapping \cite{A-1998}, \cite{FST}.


 \end{remark}

\begin{proposition} \label{maximal}  Let $r_{0}, s_{0}$ such that $1 \le r_0  \le p'-1, 1 \le s_0 \le p-1$ and  $\lambda \notin  (\lambda_{r_0,s_0} ^{\pm} + {\Z}) $.

  Assume that
\bea h - h_{p,p'} ^{r_0,s_0 } \notin {\Z}_{> 0} \qquad \forall h \in \mathcal{S}_{p,p'}.  \label{uvjet-ired} \eea
Then $\mathcal{E}_{r_0, s_0} ^{\lambda}$
is an  irreducible $V_k(sl(2))$--module. In particular, $\mathcal{E}_{r_0, s_0} ^{\lambda}$ is irreducible if $h_{p,p'} ^{r_0,s_0 } $ is maximal  in the set  $\mathcal{S}_{p,p'}$.

\end{proposition}
\begin{proof}
Assume that $\mathcal{E}_{r_0, s_0} ^{\lambda} $ is reducible. By Lemma \ref{injektivnost} the operator $e(0) = e^c_0$   acts injectively  on    the module $\mathcal{E}_{r_0, s_0} ^{\lambda} $, and therefore  there are no  submodules of   $\mathcal{E}_{r_0, s_0} ^{\lambda} $ with finite-dimensional $L_{sug}(0)$--eigenspaces. Since the  top component $\mathcal{E}_{r_0, s_0} ^{\lambda}(0)$ is an irreducible $sl(2)$--module, we conclude that $\mathcal{E}_{r_0, s_0}^{\lambda}$ has a non-trivial $\Z_{\ge 0}$--graded irreducible  subquotient $M = \bigoplus_{m=0} ^{\infty} M(m)$,  such that $\dim M(0)= \infty$.  By   Lemma  \ref{beskonacne},   the conformal weight of $M(0)$ is   $h + k/4$, such that
$$h \in \mathcal{S}_{p,p'}, \ h > h_{p,p'} ^{r_0,s_0 }, \ h - h_{p,p'} ^{r_0,s_0 }   \in  {\Z}_{> 0}. $$  This contradicts the choice of $(r_0, s_0)$. The claim  holds.
\end{proof}

\begin{remark} \label{KR-sl}
A new result on   characters of irreducible relaxed $V_k(sl(2))$ from \cite[Theorem 5.2]{KR} directly   implies that $\mathcal{E}_{r, s} ^{\lambda}$ is irreducible  for all $1 \le r  \le p'-1, 1 \le s \le p-1$ and $\lambda \notin  (\lambda_{r,s} ^{\pm} + {\Z})$.
\end{remark}

 \subsection{Realization of the intertwining operators among relaxed modules}
 We will now compare our realization with conjectural fusion rules from \cite{CR-2013}.

Note that in the realization we have  $ \pi_{\ell}  ( \mathcal{E}_{r,s} ^{\lambda} ) = L^{Vir}(d_{p,p'},h_{p,p'} ^{r,s}) \otimes \Pi_{ \ell -1} (\lambda)$ (cf. Proposition \ref{spectral-flow}). One can show that if   $\left[\begin{matrix} (r'', s'') \\ (r, s) \ (r', s')\end{matrix}\right]_{ (p,p') }=1$, we can realize the intertwining operator of the following type
 \bea \label{int-real} \binom{   \pi _{ \ell + \ell'- 1} (  \mathcal{E}_{r'', s''} ^{\lambda + \lambda'}  )}{\pi _{\ell} (  \mathcal{E}_{r, s} ^{\lambda} ) \ \  \pi _{\ell'} (  \mathcal{E}_{r', s' } ^{\lambda'}  )   }.\eea

  \begin{itemize}
\item There is a small difference between notation in \cite{CR-2013} where relaxed modules were denoted by $\mathcal E_{\mu, \Delta_{r,s}}$, where $\Delta_{r,s} = h_{p, p'} ^{r,s}$. Precise  relation is
$  \mathcal{E}_{r, s} ^{\lambda}   = \mathcal E_{2\lambda- k, \Delta_{r,s}}. $

\item  Based on   the Grothendieck fusion rules    \cite[Propositions 13 and 18]{CR-2013} and \cite[Proposition 2.17]{ACR}, it is conjectured that the following fusion rules holds in the category of weight $V_{k}(sl(2))$--modules:

\bea
&&\pi _{\ell} ( \mathcal E_{\mu , \Delta_{r,s}}  ) \boxtimes  \pi _{\ell'} (  \mathcal E_{\mu ' , \Delta_{r,s}} ) \\
  & = & \sum_{ r'', s''} \left[\begin{matrix} (r'', s'') \\ (r, s) \ (r', s')\end{matrix}\right]_{ (p,p') }\left(  \pi _{\ell + \ell' +1}  (\mathcal E_{\mu + \mu' -k , \Delta_{r'',s''}} ) + \pi_{\ell + \ell' -1}  (\mathcal E_{\mu + \mu' +k , \Delta_{r'',s''}})     \right)  \nonumber \\
&& +  \sum_{ r'', s''} \left(  \left[\begin{matrix} (r'', s'') \\ (r, s) \ (r', s'-1 )\end{matrix}\right]_{ (p,p') }    +  \left[\begin{matrix} (r'', s'') \\ (r, s) \ (r', s'+1 )\end{matrix}\right]_{ (p,p') }   \right)  \pi _{\ell + \ell'}  ( \mathcal E_{\mu + \mu' , \Delta_{r'',s''}} )\nonumber
\eea


\end{itemize}

 In free field realization, we can only construct an intertwining operator of the type (\ref{int-real}) which is in the terminology of \cite{CR-2013}:
$$ \binom{   \pi_{\ell + \ell' -1}  (\mathcal E_{\mu + \mu' +k , \Delta_{r'',s''}})}{\pi _{\ell} ( \mathcal E_{\mu , \Delta_{r,s}}  ) \ \   \pi _{\ell'} (  \mathcal E_{\mu ' , \Delta_{r,s}} )    }.$$
The construction of other three type of intertwining operators is still an open problem.

 We should mention that the cases of  the collapsing levels $k=-1/2$ and $k =-4/3$ are very interesting, since then $V_k(sl(2))$ is related with the triplet vertex algebras $W(p)$ for $p=2,3$ (cf. \cite{Ridout}, \cite{A-2005}). The fusion rules for these vertex algebras are also related to the (conjectural) fusion rules for  the singlet vertex algebra $\mathcal M(p)$ (cf. \cite{CM}, \cite{AdM-2017}).
By \cite{CR-2013},  the    fusion rules for $k=-1/2$ are
 \bea \label{fusion-12} \pi _{\ell} ( \mathcal E_{\mu , -1/8}  ) \boxtimes  \pi _{\ell'} (  \mathcal E_{\mu ' , -1/8} ) = \pi _{\ell + \ell'-1} ( \mathcal E_{\mu + \mu' -1/2  , -1/8}  ) +  \pi _{\ell + \ell'+1} ( \mathcal E_{\mu + \mu' +1/2  , -1/8}  ). \eea
In our forthcoming paper \cite{AdP}, we shall study the fusion rules (\ref{fusion-12}).

\section{Whittaker modules for $V_k(sl(2))$}
\label{Whittaker}

In this section we extend result from \cite{ALZ-2016} and construct all degenerate Whittaker modules at an arbitrary  admissible level. As a consequence, we will see that admissible affine vertex algebra  $V_k(sl(2))$   contains ${\Z}_{\ge 0}$--graded modules of the Whittaker type.

Let us first recall some notation from \cite{ALZ-2016}.

For a $(\lambda, \mu) \in {\C} ^2$, let $\widetilde {Wh}_{\widehat{sl_2}} (\lambda, \mu, k)$ denotes the universal Whittaker module at level $k$ which is generated by the Whittaker vector $w_{\l, \mu, k}$ satisfying
\bea
e(n) w_{\l, \mu, k} &=& \delta_{n,0} \lambda w_{\l, \mu, k} 	\quad (n \in {\Z}_{\ge 0}), \\
f(m) w_{\l, \mu, k} &=& \delta_{m,1} \mu w_{\l, \mu, k}	\quad (m \in {\Z}_{\ge 1}).
\eea

If $\mu \cdot \lambda \ne 0$, then the Whittaker module is called non-degenerate. It was proved in \cite{ALZ-2016} that at the non-critical level the universal non-degenerate Whittaker module is automatically irreducible. But in the degenerate case when $\mu= 0$, $\widetilde {Wh}_{\widehat{sl(2)}} (\lambda, 0 , k)$  is reducible and it contains a non-trivial submodule
$$  M _{\widehat{sl(2) }} (\lambda, 0, k, a ):= \widetilde {Wh}_{\widehat{sl(2)}} (\lambda, 0, k) / U(\widehat{sl(2)}) . (L_{sug} (0)-a) w_{\l, \mu, k} \quad ( a\in {\C}). $$
Let $ Wh _{\widehat{sl(2)}} (\lambda, 0,  k, a )$ be the simple quotient of $M _{\widehat{sl(2)}} (\lambda, 0, k, a )$.

We have the following result.

\begin{theorem} \label{whittaker-degenerate} For all $ k, h, \lambda \in {\C}$, $\lambda \ne 0$  we have:
$$ Wh _{\widehat{sl(2)}} (\lambda, 0,  k, h + k/4 ) \cong  L^{Vir} (d_k,  h) \otimes \Pi_{\lambda}. $$
\end{theorem}
\begin{proof}
The proof will use \cite[Lemma 10.2]{ALZ-2016} which says that $\Pi_{\lambda}$ is an irreducible $\widehat{\mathfrak b}_1$--module, where $\widehat{\mathfrak b}_1$ is a Lie subalgebra of $\widehat{sl(2) }$ generated by $e(n), h(n)$, $n \in {\Z}$.

 On  $L^{Vir} (d_k,  h)$ we have the weight decomposition:
 $$  L^{Vir} (d_k,  h)= \bigoplus_{m \in {\Z}_{\ge 0} } L^{Vir} (d_k,  h)_{h + m},  \quad  L^{Vir} (d_k,  h)_{h + m} = \{v \in  L^{Vir} (d_k,  h) \ \vert L(0) v = (h +m ) v \}. $$
Let $v_h$ be the highest weight vector in  $L^{Vir} (d_k,  h)$, and define
$ \widetilde w_{\l,  0,  k} = v_h \otimes w_{\lambda}$.
Since
\bea
e(n) \widetilde  w_{\l, 0, k} &=& \delta_{n,0} \lambda w_{\l, 0, k} 	\quad (n \in {\Z}_{\ge 0}), \\
f(m) \widetilde w_{\l, 0,  k} &=& 0	\quad (m \in {\Z}_{\ge 1}), \\
L_{sug} (n) \widetilde  w_{\l, 0, k} &=& \delta_{n,0}  (h + k/4) w_{\l, 0, k} 	\quad (n \in {\Z}_{\ge 0})
\eea
we conclude that  $  \widehat W= U (\widehat{sl(2) }).  \widetilde w_{\l, 0, k} $ is a certain quotient of
$M _{\widehat{sl(2)}} (\lambda, 0, k,  h + k/4 )$.

  Let us first prove that $\widehat W =  L^{Vir} (d_k,  h) \otimes \Pi_{\lambda} $.
 It suffices to prove that for every $m \in {\Z}_{\ge 0}$ we have that
 \bea
  &&  v \otimes w  \in  \widehat W \quad \forall  v \in   L^{Vir} (d_k,  h)_{h + m},  \quad  \forall  w\in \Pi_{\lambda}.  \label{claim-1}
 \eea
  For $m= 0$, the claim follows by using  the irreducibility of $\Pi_{\lambda}$  as a $\widehat{\mathfrak b}_1$--module.
  Assume now that    $v'  \otimes w   \in  \widehat W   $  for all  $v'   \in    L^{Vir} (d_k,  h)_{h + m'} $ such that  $m' < m$ and  all $w \in \Pi_{\lambda}$.
   Let $ v \in  L^{Vir} (d_k,  h)_{h + m} $. It suffices to consider homogeneous  vectors
  $$ v = L(-n_0)  L(-n_1)  \cdots L(-n_s)  v_h, \quad n_0 \ge \cdots  \ge n_s \ge 1, \ n_0 + \cdots + n_s = m.$$
  Then by inductive assumption we have that  $L(-n_1)  \cdots L(-n_s)  v_h \otimes w \in \widehat W$ for all $w \in \Pi_{\lambda}$.  By using the formulae for  the action of $f (m)$, $m \in {\Z}$, we get
  $$ f(-n_0)   (L(-n_1) \cdots L(-n_s) v_h \otimes w_{\lambda} ) = AL(-n_0)  L(-n_1) \cdots L(-n_s) v_h \otimes w_{\lambda}  + z $$
where $A \ne 0$   and
$$ z =  \sum_{i} v_i \otimes w_i, \quad v_i \in    L^{Vir} (d_k,  h)_{h + m_i ' } ,  \ m_i ' < m, \ w_i \in \Pi_{\lambda}.  $$
By using inductive assumption   we get that $z \in \widehat W$, and therefore $v \otimes w_{\lambda}  \in \widehat  W$.
Using the fact that  $\Pi_{\lambda}$ is an irreducible $\widehat{\mathfrak b}_1$--module, we get that  $v \otimes w \in \widehat  W$ or every $w \in \Pi_{\lambda}$.
The claim (\ref{claim-1}) now follows by induction.

 %
Now the irreducibility result will be a consequence of the following claim:    \bea
 && v \otimes w \ \mbox{is cyclic vector in} \   L^{Vir} (d_k,  h) \otimes \Pi_{\lambda}  \quad \forall v \in  L^{Vir} (d_k,  h)_{h + m}, \ m\in {\Z}_{\ge 0},  \ \forall w \in \Pi_{\lambda}. \label{claim-2}
 \eea
 For $m=0$, the claim  (\ref{claim-2})  follows by using  irreducibility of $\Pi_{\lambda}$  as a $\widehat{\mathfrak b}_1$--module and (\ref{claim-1}).
 Assume now that  $v \in  L^{Vir} (d_k,  h)_{h + m}$ for $m >0$. Then there is  $m_0$, $0 < m_0 \le m$ so that $L(m_0) v \ne 0$ and $L(m_0) v \in  L^{Vir} (c_k,  h)_{h + m-m_0}. $
 Since
 $$ f(m_0) (v \otimes w_{\lambda} ) = (k+2)  \lambda L(m_0) v \otimes w_{\lambda},$$
 by induction we have that  $ L(m_0) v \otimes w_{\lambda}$ is a cyclic vector. So $v \otimes w_{\lambda} $ is   also cyclic. By using again the  irreducibility of $\Pi_{\lambda}$ as  $\widehat{\mathfrak b}_1$--module, we see that  $v \otimes w  $ is cyclic for all $w \in \Pi_{\lambda}$. The proof follows.
\end{proof}

As a consequence, we shall describe the structure of
simple Whittaker module $Wh_{\widehat{sl_2}} (\lambda, 0, k,a )$ at admissible levels, and show that these modules are $V_k(sl_2)$--modules.

\begin{theorem} Assume that $k$ is admissible, non-integral, and $\lambda \ne 0$. Then we have:
\item[(1)]$ Wh _{\widehat{sl (2) }} (\lambda, 0,  k, a ) \cong   L^{Vir} (d_{p,p'},  h_{p,p'}^{r,s}) \otimes \Pi_{\lambda}, $
where $a =  h_{p,p'}^{r,s} + k / 4 $.
\item[(2)] The set
$$\{ Wh_{\widehat{sl(2) }} (\lambda, 0, k, h + k/4)  \ \vert \ h \in \mathcal S_{p,p'} \}$$  provides all irreducible  ${\Z}_{\ge 0}$--graded $V_k(sl(2))$--modules which are Whittaker $\widehat{sl(2) }$--modules.
\end{theorem}
\begin{proof}
Since $ L^{Vir} (d_{p,p'},  h) $ for   $h \in \mathcal S_{p,p'} $ is a $ L^{Vir} (d_{p,p'},  0) $--module, we conclude that  $ L^{Vir} (d_{p,p'},  h ) \otimes \Pi_{\lambda}$ is  a
$ L^{Vir} (d_{p,p'},  0) \otimes \Pi(0)$--module  and therefore a $V_k(sl(2) )$--module.

Assume that  $ Wh _{\widehat{sl(2)}} (\lambda, 0,  k, h + k/4 )  $ is a $V_k(sl(2) )$--module. We proved in Theorem \ref{whittaker-degenerate}  that   $ Wh _{\widehat{sl(2) }} (\lambda, 0,  k, h + k/4 ) \cong   L^{Vir} (d_{p,p'},  h) \otimes \Pi_{\lambda} $ for certain $h \in {\C}$ and that $L_{sug} (0) $ acts on lowest weight component as $h + \frac{k}{4}$. By using description of Zhu's algebra for $V_k(sl(2) )$  (cf. \cite{AdM-MRL}, \cite{RW2}) we see that $L(0)$ must act on lowest component as $h  \cdot \mbox{Id}$ for  $h \in \mathcal S_{p,p'}$. Therefore,  $ Wh _{\widehat{sl(2)}} (\lambda, 0,  k, h + k/4 ) \cong   L^{Vir} (d_{p,p'},  h) \otimes \Pi_{\lambda} $ for $h \in  \mathcal S_{p,p'}$. The proof follows.
\end{proof}

\section{ Screening operators and logarithmic modules for $V_k(sl(2))$ }

\label{logarithmic}
This section gives a vertex-algebraic interpretation of the construction of logarithmic modules from \cite[Section 5]{FFHST}. By using the embedding of  $V_k (sl(2) ) $
in the vertex algebra $ L ^{Vir} (d_{p,p'}, 0) \otimes \Pi(0) \subset  L ^{Vir} (d_{p,p'}, 0) \otimes V_L$, we are able to use methods
 \cite{AdM-selecta} to construct logarithmic modules for admissible affine vertex algebra $V_k(sl(2))$.  Formula for the  screening operator $S$ appeared in \cite{FFHST}.  In the case $k=-\frac{4}{3}$, the construction  of logarithmic modules reconstructs modules from \cite[Section 8]{AdM-selecta} and \cite{Gab}.

Note that the basic definitions and constructions related with logarithmic modules were discussed in Section \ref{prelimin}.

\subsection{Screening operators}

First we notice that $L^{Vir} (d_{p,p'},h_{p,p'} ^{2,1})$ is an irreducible   $L ^{Vir} (d_{p,p'}, 0)$--module generated by lowest weight vector $v_{2,1}$  of conformal
weight $$h^{2,1}:= h_{p,p'} ^{2,1}= \frac{3 p' -2p}{4p} = \frac{3}{4} k + 1. $$ Let us now consider $L ^{Vir} (c_{p,p'}, 0) \otimes \Pi(0)$--module
$$ \mathcal{M}_{2,1} = L ^{Vir} (d_{p,p'}, 0) \otimes \Pi(0) . (v_{2,1} \otimes e ^{\nu}) = L^{Vir} (d_{p,p'},h_{p,p'} ^{2,1}) \otimes \Pi_{(1)} ( - \tfrac{k}{2})  .  $$
Note that $ \mathcal{M}_{2,1}$  has integral weights with respect to $L_{sug} (0)$.
Using construction from \cite{Li}, which was reviewed  in Subsection \ref{extended},  we have the extended vertex algebra
$$ \mathcal{V} = L ^{Vir} (d_{p,p'}, 0) \otimes \Pi(0) \bigoplus \mathcal{M}_{2,1}. $$

Note also
\bea  L(-2) v_{2,1}& =&  \frac{1}{k+2} L(-1) ^2 v_{2,1}.  \nonumber \\ \
   [ L(n), ( v_{2,1} )_m  ]  &=& ( (h^ {2,1} -1) (n+1) - m )  (v_{2,1} )_{m+n} \qquad (m,n \in {\Z}).  \nonumber  \\ \
   [L(-2),  ( v_{2,1} )_{-1}   ] & = &  ( 2 - h^{2,1})    (v_{2,1} )_{-3}  \nonumber  \\ \
    [L(-2),  ( v_{2,1} )_{0}   ] & = &  ( 1 - h^{2,1})    (v_{2,1} )_{-2}  \nonumber \\ \
     [L(-2),  ( v_{2,1} )_{1}   ] & = &  -  h^{2,1}    (v_{2,1} )_{-1}  \nonumber
 \eea
Let $s = v_{2,1} \otimes e^{\nu}$. By using formula (\ref{sug})  we get
$$ L_{sug}( n) s = \delta_{n,0} s \quad (n \ge 0). $$
Therefore
$$ S = s_0= \mbox{Res}_{z} Y(s,z) $$
commute with the action of the Virasoro algebra $L_{sug}(n)$, $n \in {\Z}$.

We want to see that $S$ commutes with $\widehat{sl(2)}$--action.
The arguments  for claim  were  essentially presented in \cite{FFHST}.
The following lemma can be proved by direct calculation in lattice vertex algebras.
\begin{lemma} \label{ffhst} \cite{FFHST}
We have
\bea  s_2 f &= & 2 (k+1) v_{2,1} \otimes   e^{\nu - \frac{2}{k} (\mu -\nu ) }  \nonumber \\
s_1 f & = &  k L(-1) v_{2,1}  \otimes   e^{\nu - \frac{2}{k} (\mu -\nu ) }  + (k+2) v_{2,1} \otimes \nu(-1) e^{\nu - \frac{2}{k} (\mu -\nu ) }  \nonumber \\
s_0 f &=& Sf = 0. \nonumber
\eea

\end{lemma}

\begin{proposition} \cite{FFHST}
We have:
$$ [S, \widehat{sl(2) } ] = 0, $$
i.e., $S$ is a scrrening operator.
\end{proposition}
\begin{proof}
 Since
\bea
&& s_n e = s_n h = 0  \quad (n \ge 0), \nonumber \eea
we get
$ [ S, e(n)] = [S, h(n)] = 0$.

By using Lemma \ref{ffhst} we get
$ [S, f(n)] =  (S f) (n) = 0$.
The claim follows.
\end{proof}

\subsection{Construction of logarithmic modules for $V_k(sl(2))$}
\begin{lemma}  \label{moduli-ext}
Assume that
 $\ell \in {\Z}$, $ 1\le s \le p'-1$ , $ 1 \le r \le p-2$ and
\bea \label{uvjet}
\lambda \equiv  \lambda_{r,s} ^+ = \frac{1}{2}  \left(s - 1 - (k+2)  (r-1)  \right)  \quad   \mbox{mod}( {\Z}).
 \eea
 Then we have:
 \begin{itemize}
\item[(1)]
$\mathcal M^{\ell, +} _{r,s} (\lambda) =  L^{Vir} (d_{p,p'}, h_{p,p'} ^{r,s}) \otimes \Pi_{(\ell ) } ( \lambda) \bigoplus   L^{Vir} (d_{p,p'}, h_{p,p'} ^{r+1, s}) \otimes    \Pi_{(\ell +1)} ( - \frac{k}{2} +  \lambda) $
is a $\mathcal V$--module.

\item[(2)] $S^2 = 0 $ on $\mathcal M^{\ell,+} _{r,s} (\lambda)$.

\item[(3)]   Let  $\lambda = \lambda_{r,s} ^+ -1$. Then $$ S  (v_{r,s} \otimes e^{\ell \mu + {\lambda} c} ) =   C v_{r+1, s} \otimes  e^{(\ell+1 )  \mu + (\lambda-k/2) c} , $$ where $C \ne 0$. In particular, $ S \ne 0$ on $ \mathcal M^{\ell,+} _{r,s} (\lambda)$.

\end{itemize}
\end{lemma}
\begin{proof}
In general, $L^{Vir} (d_{p,p'}, h_{p,p'} ^{r,s}) \otimes \Pi_{(\ell ) } ( \lambda) $ is a ${\Z}$--graded module whose conformal weights are congruent  $\mbox{mod} (\Z)$ to
 $ h_{p,p'}^{r,s} + \frac{1}{4} (k \ell ^ 2  +  4 (\ell +1 )  \lambda   ) $.  By direct calculation we see that
 $$ h_{p,p'} ^{r+1, s}   + \frac{1}{4} (( \ell +1)  ^ 2 k + 4 (\ell +2 ) (\lambda -\frac{k}{2}) )  \equiv  h_{p,p'}^{r,s} + \frac{1}{4} (\ell ^ 2 k + 4 (\ell +1 ) \lambda )  \qquad \mbox{mod} (\Z)$$
if and only if  (\ref{uvjet}) holds.  Therefore
we conclude    $L^{Vir} (d_{p,p'}, h_{p,p'} ^{r,s}) \otimes \Pi_{(\ell ) } ( \lambda) $ and $ L^{Vir} (d_{p,p'}, h_{p,p'} ^{r,s+1}) \otimes    \Pi_{(\ell +1)} ( - \frac{k}{2} +  \lambda) $ have conformal weights congruent
 $\mbox{mod} (\Z)$ if and only if    (\ref{uvjet}) holds.

 Let $\mathcal Y_1( \cdot, z)$ be the non-trivial intertwining operator of type
$$\binom{  L^{Vir} (d_{p,p'}, h_{p,p'} ^{r+1,s})}{L^{Vir} (d_{p,p'},h_{p,p'} ^{2,1})    \ \
L^{Vir} (d_{p,p'}, h_{p,p'} ^{r,s})  }  $$
Then $$ \mathcal Y_1 ( v , z) = \sum_{ r \in {\Z} + \Delta} v _r z ^{-r -1}  \quad (v \in   L^{Vir} (c_{p,p'},h_{p,p'} ^{2,1}))$$
where  $$\Delta = h_{p,p'} ^{2,1}+ h_{p,p'} ^{r,s} -  h_{p,p'} ^{r+1,s}  = \frac{1}{2} ( (s-1) - (r-1)(k+2)) = {\lambda}_{r,s} ^+ .$$
In particular we have  $$ ( v_{2,1} )_{\Delta-1} v_{r,s} = C_2 v_{r+1,s} \quad (C_2 \ne 0),   ( v_{2,1} )_{\Delta + n } v_{r,s} =  0 \quad (n \in {\Z}_{\ge 0}). $$

  Let $\mathcal Y_2( \cdot, z)$ be the non-trivial intertwining operator of type
$$\binom{  \Pi_{(\ell +1)} ( - \frac{k}{2} +  \lambda)}{\Pi_{(1)} ( - \tfrac{k}{2})     \ \
 \Pi_{(\ell ) } ( \lambda)    }  $$
Then $$ \mathcal Y_2 (v  , z) = \sum_{ r \in {\Z} - \lambda} v _r z ^{-r -1} \quad (v \in   \Pi_{(1)} ( - \tfrac{k}{2}) = \Pi(0) . e^{\nu})$$
In particular we have \bea  &&e^{\nu}_{-\lambda -1}   e^{\ell \mu + \lambda c} =  C_1 e^{(\ell+1 )  \mu + (\lambda-k/2) c}  \quad (C_1 \ne 0) \nonumber  \\
&&e^{\nu}_ {-\lambda -n-1 }   e^{\ell \mu + \lambda c}  \ne 0 ,  \quad e^{\nu}_{-\lambda + n }   e^{\ell \mu + \lambda c}  =0  \quad  (n \in {\Z_{\ge 0}} ) \nonumber \eea

 We conclude that there is an  non-trivial intertwining operator
$\mathcal Y = \mathcal Y_1 \otimes \mathcal Y_2  $  of type $$\binom{  L^{Vir} (c_{p,p'}, h_{p,p'} ^{r+1,s}) \otimes    \Pi_{(\ell +1)} ( - \frac{k}{2} +  \lambda)}{L^{Vir} (c_{p,p'},h_{p,p'} ^{2,1}) \otimes \Pi_{(1)} ( - \tfrac{k}{2})  \ \
L^{Vir} (c_{p,p'}, h_{p,p'} ^{r,s}) \otimes \Pi_{(\ell ) } ( \lambda) }$$
 with integral powers of $z$. Now, the  assertion (1)  follows by applying Lemma \ref{AdM-2012}.

 By construction  we have    $S ^2 = 0 $ on  $\widetilde { \mathcal M^{\ell, +} _{r,s} (\lambda) }$, so (2) holds.

 For  $\lambda = {\lambda_{r,s} ^+} -1 $ we have
\bea S  v_{r,s} \otimes  e^{\ell \mu + \lambda c} &=&    ( v_{2,1} )_{\Delta-1} v_{r,s}   \otimes   e^{\nu}_{-\lambda -1}   e^{\ell \mu + \lambda c}    \nonumber \\
&=& C_1 \cdot C_2 v_{r+1, s} \otimes  e^{(\ell+1 )  \mu + (\lambda-k/2) c}  \ne 0. \eea
The proof follows.
\end{proof}

By using the Virasoro  intertwining operator of type $$\binom{  L^{Vir} (d_{p,p'}, h_{p,p'} ^{r-1,s})}{L^{Vir} (d_{p,p'},h_{p,p'} ^{2,1})    \ \
L^{Vir} (d_{p,p'}, h_{p,p'} ^{r,s})  }  $$ and an analogous  proof to that of Lemma   \ref{moduli-ext} we get:
\begin{lemma}  \label{moduli-ext-2}
Assume that
 $\ell \in {\Z}$, $ 1\le s \le p'-1$ , $ 2 \le r \le p-1$ and
\bea
\lambda \equiv  \lambda_{r,s} ^- = -\frac{1}{2}  \left(s + 1 - (k+2)  (r+1)  \right)  \quad   \mbox{mod}( {\Z}).
 \eea
 Then we have:
 \begin{itemize}
\item[(1)]
$\mathcal M^{\ell,-} _{r,s} (\lambda) =  L^{Vir} (d_{p,p'}, h_{p,p'} ^{r,s}) \otimes \Pi_{(\ell ) } ( \lambda) \bigoplus   L^{Vir} (d_{p,p'}, h_{p,p'} ^{r-1, s}) \otimes    \Pi_{(\ell +1)} ( - \frac{k}{2} +  \lambda) $
is a $\mathcal V$--module.

\item[(2)] $S^2 = 0 $ on $\mathcal M^{\ell, -} _{r,s} (\lambda)$.

\item[(3)]   Let  $\lambda = \lambda_{r,s} ^- -1$. Then $$ S  (v_{r,s} \otimes e^{\ell \mu + {\lambda} c} ) =   C v_{r-1, s} \otimes  e^{(\ell+1 )  \mu + (\lambda-k/2) c} , $$ where $C \ne 0$. 

\end{itemize}
\end{lemma}


Next result shows that at a non-integral admissible levels, logarithmic modules always exist. By applying  Theorem \ref{AdM-selecta} and taking $v=s$ we get:
\begin{proposition}
Assume that $(\mathcal M, \mathcal Y_{\mathcal M})$ is any $\mathcal V$--module. Then
$$ (\widetilde {\mathcal M},\widetilde { \mathcal Y_{\mathcal M}} (\cdot, z)  ) := (\mathcal M, \mathcal Y_{\mathcal M} (\Delta(s,z) \cdot, z))$$
is a $V_k(sl(2))$--module such that
$$ \widetilde L_{sug} (0) = L_{sug} (0) + S. $$

In particular,  $\widetilde {\mathcal V}$ is a logarithmic $V_k(sl(2))$--module of $ \widetilde L_{sug} (0)$ nilpotent rank   two.


\end{proposition}
\begin{proof}
First claim follows directly by  applying  Theorem \ref{AdM-selecta}.

The assertion (1) follows from the following observations:
\begin{itemize}
\item $L_{sug}(0)$ acts semi-simply on $\mathcal V$, and $\widetilde L_{sug} (0) - L_{sug} (0)  = S$ on  $\widetilde {\mathcal V}$.
\item By construction $ S (  L ^{Vir} (d_{p,p'}, 0) \otimes \Pi(0)   )  \subset \mathcal M_{2,1}$, $ S (\mathcal M_{2,1} )= 0$, so $S^2 = 0$ on $\mathcal V$.
Since $ S \nu (-1) {\bf 1} = \frac{k}{2}v_{2,1} \otimes e^{\nu} \ne 0 $ we have
$$ ( \widetilde L_{sug} (0) - L_{sug} (0)  )  \ne 0, ( \widetilde L_{sug} (0) - L_{sug} (0)  )  ^2 = 0 \quad \mbox{on } \  \widetilde {\mathcal V}. $$
\end{itemize}
\end{proof}

Using Lemma \ref{moduli-ext} we obtain:

\begin{corollary}   \label{cor-log-1}  Assume that
 $\ell \in {\Z}$, $ 1\le s \le p'-1$ , $ 1 \le r \le p-2$ and $\lambda = \lambda_{r,s} ^+$.  Then we have:
 \begin{itemize}
 \item[(1)]  $\widetilde { \mathcal M^{\ell, +} _{r,s} (\lambda ) }$ is a logarithmic $V_k(sl(2))$--module of $ \widetilde L_{sug} (0)$ nilpotent rank   two,
 \item[(2)] The logarithmic module $\widetilde { \mathcal M^{\ell, +} _{r,s} (\lambda  ) }$ appears in the following  non-split extension of weight $V_k (sl(2))$--modules:
 $$ 0 \rightarrow     L^{Vir} (d_{p,p'}, h_{p,p'} ^{r+1, s})    \otimes \Pi_{(\ell +1)} ( - \frac{k}{2} +  \lambda)  \rightarrow \widetilde { \mathcal M^{\ell, +} _{r,s} (\lambda) } \rightarrow       L^{Vir} (d_{p,p'}, h_{p,p'} ^{r,s}) \otimes \Pi_{(\ell ) } ( \lambda)  \rightarrow 0. $$
 \end{itemize}
\end{corollary}
\begin{proof}
  Lemma \ref{moduli-ext}  gives that $S \ne 0$ and $S ^2 = 0$ on   $\mathcal M^{\ell} _{r,s} (\lambda )$.

This implies that
 $$ ( \widetilde L_{sug} (0) - L_{sug} (0)  )  \ne 0, ( \widetilde L_{sug} (0) - L_{sug} (0)  )  ^2 = 0 \quad \mbox{on } \  \widetilde { \mathcal M^{\ell, +} _{r,s} (\lambda ) },$$
 which proves the assertion (1).

 Note that    $L^{Vir} (d_{p,p'}, h_{p,p'} ^{r+1, s})    \otimes \Pi_{(\ell +1)} ( - \frac{k}{2} +  \lambda)  $ is a submodule of $\widetilde { \mathcal M^{\ell, +} _{r,s} (\lambda ) }$ on  which $S$ acts trivially. Therefore  $L^{Vir} (d_{p,p'}, h_{p,p'} ^{r+1, s})    \otimes \Pi_{(\ell +1)} ( - \frac{k}{2} +  \lambda)  $  is a weight submodule. The quotient module is isomorphic to  the weight module $L^{Vir} (d_{p,p'}, h_{p,p'} ^{r,s}) \otimes \Pi_{(\ell ) } ( \lambda) $.

 Since $\widetilde { \mathcal M^{\ell, +} _{r,s} (\lambda) }$ is non-weight module by (1), we have that the sequence  in (2) does not split. The proof follows.
\end{proof}

Similarly, using   Lemma \ref{moduli-ext-2} we obtain:

\begin{corollary}   \label{cor-log-2}     Assume that
 $\ell \in {\Z}$, $ 1\le s \le p'-1$ , $ 2 \le r \le p-1$ and $\lambda = \lambda_{r,s} ^-$.  Then we have:
 \begin{itemize}
 \item[(1)]  $\widetilde { \mathcal M^{\ell, -} _{r,s} (\lambda ) }$ is a logarithmic $V_k(sl(2))$--module of $ \widetilde L_{sug} (0)$ nilpotent rank   two,
 \item[(2)] The logarithmic module $\widetilde { \mathcal M^{\ell, -} _{r,s} (\lambda  ) }$ appears in the following  non-split extension of weight $V_k (sl(2))$--modules:
 $$ 0 \rightarrow     L^{Vir} (d_{p,p'}, h_{p,p'} ^{r-1, s})    \otimes \Pi_{(\ell +1)} ( - \frac{k}{2} +  \lambda)  \rightarrow \widetilde { \mathcal M^{\ell,-} _{r,s} (\lambda) } \rightarrow       L^{Vir} (d_{p,p'}, h_{p,p'} ^{r,s}) \otimes \Pi_{(\ell ) } ( \lambda)  \rightarrow 0. $$
 \end{itemize}
\end{corollary}

\section{A realization of  the $N=3$ superconformal vertex algebra  $W_{k'} (spo(2,3), f_{ \theta} )$ for $k'=-1/3$ }
\label{real-N3}
The cases $k \in \{-1/2, -4/3\}$ were already studied in the literature. In these cases the quantum Hamiltonian reduction maps $V_k(sl(2))$ to the trivial vertex algebra, and therefore  the affine vertex algebra $V_k(sl(2)$ is realizead as a vertex subalgebra of $\Pi(0)$. In the case $k=-1/2$, $V_k(sl(2))$ admits a realization as a subalgebra of the Weyl vertex algebra and it is also related with the triplet vertex algebra $\mathcal W(2)$ with central  charge $c=-2$ (cf. \cite{FF}, \cite{Ridout}). In \cite{A-2005}, the author related $V_{-4/3}(sl(2))$ with the triplet vertex algebra $\mathcal W(3)$ at central charge $c=-7$.
\vskip 5mm

Let $k=-2/3$.  Then $V_k(sl(2))$ is realized as a subalgebra of $L^{Vir} (d_{3,4},0) \otimes \Pi (0)$. But $L^{Vir} (d_{3,4}, 0)$ is exactly the even subalgebra of the fermionic vertex superalgebra  $F = L^{Vir} (d_{3,4},0) \oplus L^{Vir} (d_{3,4},\frac{1}{2})$ generated by the odd field $\Psi(z) = \sum_{m \in {\Z}} \Psi(m + \tfrac{1}{2}) z^{-m-1}$ (see Section \ref{osp(1,2)}).  The Virasoro vector is $\omega_F = \frac{1}{2}  \Psi(-\frac{3}{2}) \Psi (-\frac{1}{2}) {\bf 1}$.  Let $\gamma  = \frac{2}{k}  \nu$,  $\varphi = \frac{2}{k} \mu$. Then
$$ \la \gamma, \gamma \ra = - \la \varphi , \varphi \ra = 3. $$
Let $D ={\Z} \gamma$. Then $V_D =M_{\gamma} (1) \otimes {\C}[D]$ is the lattice vertex superalgebra, where  $M_{\gamma}(1)$ is the Heisenberg vertex algebra generated by $\gamma$ and $ {\C}[D]$ the group algebra of the lattice $D$.
The screening operator $S$ is then expressed as
$$    S =  \mbox{Res}_z :\Psi (z) e^{\nu } (z): =\mbox{Res}_z :\Psi (z) e^{-\frac{1} {3}  \gamma} (z):. $$
Define also  $$  Q = \mbox{Res}_z :\Psi (z) e^{\gamma} (z): $$  We have:
\begin{proposition} \cite{AdM-supertriplet}
\item[(1)] $ SW(1) \cong \mbox{Ker}_{ F \otimes V_D} S $
is isomorphic to the $N=1$ super-triplet vertex algebra  at central charge $c=-5/2$  strongly  generated  by  $$ X = e^{-\gamma}, \  H =  Q X , \   Y = Q^2 X,
\ \widehat X = \Psi(-1/2) X, \  \widehat H = Q \widehat X,  \ \widehat Y = Q^2 \widehat X $$
and superconformal vector $\tau = \frac{1}{\sqrt{3}} \left( \Psi (-\frac{1}{2})  \gamma (-1)   + 2 \Psi(-\frac{3}{2})  \right){\bf 1}  $ and corresponding conformal vector
$$ \omega_{N=1}  = \frac{1}{2} \tau_0 \tau = \frac{1}{6} \left ( \gamma(-1) ^2  + 2 \gamma (-2) \right) {\bf 1} + \frac{1}{2} \Psi(-\frac{3}{2}) \Psi(-\frac{1}{2}) {\bf 1}.  $$

\item[(2)] $ \overline{SM(1)} \cong \mbox{Ker}_{ F \otimes M_{\gamma} (1) } S $
is isomorphic to the $N=1$ super singlet vertex algebra at central charge $c=-5/2$  strongly  generated by $\tau, \omega_{N=1}  , H, \widehat H$.

\end{proposition}

 Consider the lattice vertex algebra $F_{-3} = V_{\Z \varphi}$. We shall now see that the admissible affine vertex algebra $V_{-2/3} (sl (2))$ is a vertex subalgebra of $SW(1) \otimes F_{-3}$. Note that $\gamma(0) - \varphi (0)$ acts semisimply on $SW(1) \otimes F_{-3}$ and we have the following vertex algebra
 $$ \mathcal U = \{ v \in SW(1) \otimes F_{-3} \vert \  ( \gamma(0) - \varphi (0) ) v = 0 \}. $$

 Following \cite{KW04}, we identify the $N=3$ superconformal vertex algebra with affine $W$--algebra  $W_{k'} (spo(2,3), f_{ \theta} )$. By applying  results on   conformal embeddings from  \cite[Theorem 6.8 (12)]{AKMPP-JJM},  we see that  the  vertex algebra $W_{k'} (spo(2,3), f_{ \theta} )$ for $k'=-1/3$  is  isomorphic  to the simple current extension of $V_{-2/3} (sl(2))$:
 $$W_{k'} (spo(2,3),f_{ \theta} ) =   L_{A_1} (-\frac{2}{3} \Lambda_0) \bigoplus  L_{A_1} (-\frac{8}{3} \Lambda_0 + 2 \Lambda_1). $$

 \begin{theorem} We have:
 \item[(1)] $\mathcal U \cong W_{k'} (spo(2, 3), f_{\theta})  $ for $k' = -1/3$.
 \item[(2)] $ \mbox{Com} (M_h (1) , W_{k'} (spo(2,3),f_{ \theta} ) ) \cong  \overline{SM(1)} $.
 \item[(3)] $\mbox{Ker} (sl(2), -\frac{2}{3} ) \cong \overline{SM(1)} ^0$, where $\overline{SM(1)} ^0$ is even subalgebra of the supersinglet vertex algebra $\overline{SM(1)}$.
 \end{theorem}
\begin{proof}
Since
\bea  Y  & =&  Q ^2 X = \left( - 6 \Psi(-\frac{3}{2}) \Psi (-\frac{1}{2})  + \gamma(-1) ^2 - \gamma(-2) \right) e^{\gamma} \nonumber \\
 & = & - 9 \left(  (k+2) \omega_F  - \nu (-1) ^2 - (k+1) \nu (-2) \right) e^{\gamma} \quad (k=-2/3) \nonumber \eea
we have that
\bea
 e &=& X \otimes  e^{\varphi}  = e^{\varphi - \gamma},  \nonumber \\
 h &=&  -\frac{2}{3} \varphi,   \nonumber \\
 f &=& - \frac{1}{9}  Y \otimes e^{-\varphi}  =  - \frac{1}{9} Q^2 e^{- \varphi - \gamma},  \nonumber  \\
 \omega_{sug} &=& \omega_{N=1} -\frac{1}{6} \varphi (-1) ^2 {\bf 1}. \nonumber
\eea
This implies that $V_{-2/3} (sl(2))$ is a vertex subalgebra of $\mathcal U$.  Therefore $\mathcal U$ is a $V_{-2/3} (sl(2))$--module which is $\frac{1}{2} {\Z}_{\ge 0}$--graded
$ \mathcal U = \bigoplus _{m \in  \frac{1} {2} {\Z} _{\ge 0}  } \mathcal U _ m$ with respect to $L_{sug} (0)$.
One directly sees that $\mathcal U_{1/2} = \{ 0\}$ and that
$\mathcal U_{3/2} = \mbox{span}_{\C} \{  \widehat{X}\otimes  e^{-\varphi}, \tau ,  \widehat{Y}\otimes  e^{-\varphi} \}$.  Then $\mathcal U_{3/2}$ generates a $V_{-2/3} (sl(2))$--module isomorphic to  $L_{A_1} (-\frac{8}{3} \Lambda_0 + 2 \Lambda_1)$. Since $\mathcal U$ is completely reducible  as $V_{-2/3} (sl(2))$--module we easily conclude that
$$ \mathcal U \cong   L_{A_1} (-\frac{2}{3} \Lambda_0) \bigoplus  L_{A_1} (-\frac{8}{3} \Lambda_0 + 2 \Lambda_1). $$
Since $\mathcal U$ is simple and   the extension of $V_{-2/3} (sl(2))$ by  its simple current module $L_{A_1} (-\frac{8}{3} \Lambda_0 + 2 \Lambda_1)$ is unique, we get  the assertion (1).
Assertion (2) follows from
 \bea   \mbox{Com} (M_h (1) , W_{k'} (spo(2,3),f_{ \theta} ) )  &=&  \{ v \in SW(1) \otimes F_{-3}  \ \vert \  \varphi(n) v =  (\gamma(0) - \varphi (0) ) v = 0, \ n \ge 0 \}  \nonumber \\
 &\cong & \mbox{Ker} _{SW(1)} \gamma (0) = \overline{SM(1)}.  \nonumber \eea
 (3) easily follows from (2).
\end{proof}

 \section{Realization of $V_k(osp(1,2)) $}
 \label{osp(1,2)}

 A free field realization of $\widehat{osp(1,2)}$ of the Wakimoto type was presented in \cite{ERSS}.
 In this section we study  an explicit realization of affine vertex algebras associated to $\widehat{osp(1,2)}$ which generalize realizations for $\widehat{sl(2)}$ from previous sections.
 Since the quantum Hamiltonian reduction of $V^k (osp(1,2)) $ is the  $N=1$ Neveu-Schwarz vertex algebra $V^{ns} (c_k,0)$ where $c_k= \frac{3}{2} - 12 \frac{ (k+1) ^2 }{2k+3}$ (cf. \cite[Section 8.2]{KW04}), one  can expect that inverse  of the quantum Hamoltonian reduction  (assuming that it should  exist)  gives a realization of the form
  $V^{ns} (c_k,0) \otimes \mathcal F$, where $\mathcal F$ is a certain vertex algebra of  free-fields. In this section we show that for $\mathcal F$ one can take  the tensor product of the fermionic vertex algebra $F$ at central charge $1/2$ and the lattice type vertex algebra $\Pi^{1/2}(0)$ introduced in Section \ref{pi0}.

\subsection{ Affine vertex algebra $V_k(osp(1,2)) $}

Recall that $\g = osp(1,2)$ is the simple complex Lie superalgebra   with basis $\{e,f, h, x, y\}$ such that the even part
$\g^{0} = \mbox{span}_{\C} \{ e, f, h \}$ and the  odd part $\g^{1}  =\mbox{span} _{\C} \{x, y \}$.
The anti-commutation relations are given by
\bea
 && [e,f] = h, \ [h, e] = 2 e, \ [h, f] = - 2 f \nonumber \\
 && [h, x] = x, \ [ e, x] = 0,  \ [f, x] = -y \nonumber \\
 && [h, y] = -y, \ [e, y] =-x \ [f, y] = 0 \nonumber \\
 && \{ x, x\} = 2 e, \ \{x, y\} = h, \ \{ y, y \} = -2 f. \nonumber
\eea
Choose the non-degenerate super-symmetric bilinear form  $ ( \cdot , \cdot )$ on $\g$ such that non-trivial products are given by
$$ (e, f) = (f,e) = 1, \ (h,h) = 2, \ (x, y) =  -(y, x) = 2. $$
Let ${\widehat \g} = {\g} \otimes {\C}[t, t^{-1}] + {\C} K$ be the associated affine Lie superalgebra, and $V^k (\g)$ (resp. $V_k (\g)$) the associated universal (resp. simple) affine vertex algebra.
As usually, we identify  $x \in {\g}$ with $x(-1) {\bf 1}$.

 \subsection{Clifford vertex algebras}  Consider the Clifford algebra $Cl$  with generators $\Psi _i  (r)$, $r \in {\hf} + {\Z}$, $i =1, \cdots, n$  and relations
$$ \{ \Psi  _i(r), \Psi_j  (s)\} =  \delta_{r+s, 0} \delta_{i,j}  , \quad  ( r,s  \in {\hf} + {\Z}, \ 1\le i, j \le n) . $$
Then the fields
$$  \Psi_i  (z) = \sum _{m \in {\Z} } \Psi_i (m+\tfrac{1}{2}) z^{-m-1} \quad (i=1, \dots, n) $$ generate on $F_n= \bigwedge  (\Psi_i  (-n-1/2) \ \vert \ n \in {\Z}_{\ge 0} )$ a unique structure of the  vertex superalgebra with conformal vector
$$\omega_{F_n}  =  \sum_{i=1} ^n \frac{1}{2} \Psi_i(-\tfrac{3}{2})  \Psi_i(-\tfrac{1}{2}) {\bf  1}$$
of central charge  $n /2$. Let $F = F_1$. Then $F$ is a vertex operator superalgebra of central charge $c=d_{3,4} = 1/2$.  Moreover $F = F^{even} \oplus F^{odd}$ and
$$ F^{even} = L^{Vir} (d_{3,4}, 0), \quad F^{even} = L^{Vir} (d_{3,4}, 1/2). $$
Let $\sigma$ be the canonical automorphism of $F$ of order two.   The vertex algebra $F$ has precisely two irreducible $\sigma$--twisted modules $M^{\pm}$. Twisted modules can be also constructed explicitly as an exterior algebra
$$ M^{\pm} = \bigwedge \left(\Psi (-n) \ \vert \ n \in {\Z}_{>0} \right) $$
which is an irreducible module for  twisted Clifford algebra $CL^{tw}$ with generators $\Psi   (r)$, $r \in  {\Z}$, and relations $$ \{ \Psi  (r), \Psi_j  (s)\} =  \delta_{r+s, 0} , \quad  ( r, s  \in  {\Z}) . $$
$\Psi (0)$ acts on top component of  $M^{\pm}$ as $\pm \frac{1}{\sqrt{2}} \mbox{Id}$.

As a $ L^{Vir} (d_{3,4}, 0)$--module, we have that $M^{\pm} \cong L^{Vir} (d_{3,4}, \frac{1}{16} )$.

Note also that the character of $M^{\pm}$ is given by
\bea \label{weber}  \mbox{ch}[M^{\pm}](q) =\frac{f_2(\tau)}{\sqrt{2}} = q^{\frac{1}{24}} \prod_{n =1} ^{\infty} (1+q^n) \eea
where $f_2(\tau) =  \sqrt{2} q^{\frac{1}{24}}  \prod_{n =1} ^{\infty} (1+q^n)$ is the  Weber function.

\subsection{The general case}
Let $V^{ns} (c_{p,q}, 0)$ be the universal $N=1$ Neveu-Schwarz vertex superalgebra with central charge $c_{p,q} = \frac{3}{2}( 1  - \frac{ 2 (p-q) ^2}{pq} ) $.  Let $L^{ns} (c_{p,q},0)$ be its simple quotient.
If
   \bea p, q\in {\Z},  \ p, q \ge 2,\  (\frac{p-q}{2} , q) = 1 \label{uvjet-minimal},  \eea
   then $L^{ns} (c_{p,q},0)$  is called a minimal $N=1$ Neveu-Schwarz vertex operator superalgebra. It is a rational vertex operator superalgebra \cite{A-1997}.

\begin{proposition} \label{ns-vir-min} Let $p, q \in {\Z}$,  $p, q \ne 0$.  We have:
\item[(1)]  Assume that $p, q, p + q \ne 0$. The Virasoro vertex operator  algebra $V^{Vir} (d_{p, \frac{p+q}{2} }, 0) \otimes V^{Vir} (d_{\frac{p+q}{2} , q}, 0) $ is conformally embedded in $V^{ns} (c_{p,q}, 0)\otimes F$ and
$ \omega_{p,q}   + \omega_{F} =  \omega_{p, \frac{p+q}{2}} + \omega_{\frac{p+q}{2}, q} $
where
$$      \omega_{\frac{p+q}{2} ,q }   = \frac{p}{p+q} \omega_{p,q}  + i  \frac{\sqrt{pq}} { p+q}  G(-\tfrac{3}{2} )\Psi(-\frac{1}{2}){\bf 1}   + \frac{2q-p} {p+q}  \omega_F  $$
$$      \omega_{p, \frac{p+q}{2}  }   = \frac{q}{p+q} \omega_{p,q}  - i     \frac{\sqrt{pq}} { p+q}  G(-\tfrac{3}{2} )\Psi(-\frac{1}{2}){\bf 1}   + \frac{2p-q} {p+q}  \omega_F  $$

\item[(2)]  Assume that $ p + q  = 0$.  Then   $$  t_{p,q} := \frac{1}{2} \left( - \omega_{p,q}   -     G(-\frac{3}{2}) \Psi(-\frac{1}{2}){\bf 1} +  3  \omega_F  \right) \in L^{ns} (c_{p, q}, 0) \otimes F $$
 is a commutative vector in  the vertex algebra $  L^{ns} (\frac{27}{2} , 0)\otimes F. $ The vertex subalgebra generated by $t_{p,q}$ is isomorphic to the commutative vertex algebra $M_T(0)$.

\item[(3)]  Assume that $p,q$ satisfy condition (\ref{uvjet-minimal}). Then the  rational  Virasoro vertex operator algebra  $L^{Vir} (d_{p, \frac{p+q}{2} }, 0) \otimes L^{Vir} (d_{\frac{p+q}{2} , q}, 0) $  is conformally embedded in $L^{ns} (c_{p,q}, 0)\otimes F$.
\end{proposition}
\begin{proof}
Assertions  (1) follows  by direct calculation. One can also directly show that $T=t_{p,q}$ is a commutative vector in $  L^{ns} (\frac{27}{2} , 0)\otimes F $. Let $\langle T \rangle$ be the vertex subalgebra generated by $T$.
By using fact that vectors
$$L(-n_1) \cdots L(-n_r) {\bf 1}\qquad  (n_1 \ge \cdots n_r \ge 2)$$
are linearly independent in $L^{ns} (\frac{27}{2} , 0)$  and expression for $T$ one can easily  show  that the  vectors
$$T(-n_1) \cdots T(-n_r) {\bf 1} \qquad  (n_1 \ge \cdots n_r \ge 2)$$
provide a basis of $\langle T \rangle$. So $\langle T \rangle\cong M_T(0)$.

Assertion (3) was proved in  \cite{A-2004} (see also \cite{La}, \cite{LY}).
\end{proof}

\begin{theorem}   \label{thm-reali-osp} Assume that   $k + \frac{3}{2} = \frac{p}{2q} \ne 0$.
There exists  a non-trivial vertex superalgebra homomorphism
$$ \overline \Phi : V^ {k} (osp(1,2)) \rightarrow V^{ns} (c_{p, q}, 0) \otimes  F \otimes \Pi ^{1/2} (0) $$
such that
 \bea\
e & \mapsto & e^{\tfrac{2}{k} (\mu - \nu) }, \nonumber  \\
h & \mapsto & 2 \mu(-1),  \nonumber \\
f & \mapsto & \left[   \Omega_{p,q}    -\nu(-1)^{2} -  (k+1) \nu(-2) \right] e^{-\tfrac{2}{k} (\mu - \nu)} \nonumber \\
x & \mapsto &   \sqrt{2}   \Psi (-\frac{1}{2})    e^{\tfrac{1}{k} (\mu - \nu) } \nonumber \\
y & \mapsto & \sqrt{2}   \left[ - \frac{ \sqrt{-2k- 3} }{2 }    G (-3/2) +   \Psi (-\frac{1}{2})  \nu (-1) + \frac{2k+1}{2}   \Psi (-\frac{3}{2})  \right] e^{-\tfrac{1}{k} (\mu - \nu)} , \quad    \nonumber
\eea
where $\Omega_{p,q} = (k+2) \omega_{\frac{p+q}{2} ,q}$ if $k \ne -2$ and $\Omega_{p,q} = t_{p,q}$ if $k =-2$.
\end{theorem}
The proof of Theorem \ref{thm-reali-osp} will be presented in Section \ref{realizacija-proof}.

\begin{theorem}  \label{exp-osp-adm} Assume that   $k + \frac{3}{2} = \frac{p}{2q} \ne 0$ and that  $p,q$ satisfy condition (\ref{uvjet-minimal}).
\item[(1)]There exists  a non-trivial vertex superalgebra homomorphism
$$ \overline \Phi : V_{k} (osp(1,2)) \rightarrow L^{ns} (c_{p, q}, 0) \otimes  F \otimes   \Pi ^{1/2} (0) $$
such that
 \bea\
e & \mapsto & e^{\tfrac{2}{k} (\mu - \nu) }, \nonumber  \\
h & \mapsto & 2 \mu(-1), \nonumber \\
f & \mapsto & \left[    (k+2) \omega_{\frac{p+q}{2} ,q}     -\nu(-1)^{2} -  (k+1) \nu(-2) \right] e^{-\tfrac{2}{k} (\mu - \nu)} \nonumber \\
x & \mapsto &   \sqrt{2}   \Psi (-\frac{1}{2})    e^{\tfrac{1}{k} (\mu - \nu) } \nonumber \\
y & \mapsto & \sqrt{2}   \left[ - \frac{ \sqrt{-2k- 3} }{2 }    G (-3/2) +   \Psi (-\frac{1}{2})  \nu (-1) + \frac{2k+1}{2}   \Psi (-\frac{3}{2})  \right] e^{-\tfrac{1}{k} (\mu - \nu)} , \quad    \nonumber
\eea
\item[(2)]   $\omega_{p, \frac{p+q}{2}}  \in \mbox{Com} (V_k(sl(2)), V_{k} (osp(1,2)) $.
\end{theorem}
\begin{proof}
(1) Using  Theorem \ref{thm-reali-osp} we get a homomorphism $\widetilde \Phi :  V^k (osp(1,2)) \rightarrow  L^{ns} (c_{p, q}, 0) \otimes  F \otimes   \Pi ^{1/2} (0). $
Then Proposition  \ref{ns-vir-min}  implies that  $\omega_{\frac{p+q}{2} ,q} $ generates a subalgebra of $L^{ns} (c_{p, q}, 0) \otimes  F$ isomorphic to the minimal Virasoro vertex algebra $L^{Vir} ( d_{\frac{p+q}{2},q}, 0)$.
Therefore  Theorem \ref{real-simpl-sl2}  gives that $e,f,h$  generate the simple admissible affine vertex algebra $V_k(sl(2))$.

At admissible level  $k$, the vertex algebra $V^k (osp(1,2))$ contains a unique singular vector, i.e., the maximal ideal of  $V^k (osp(1,2))$ is simple. So we have two possibilities:
$$ \mbox{Im} (\widetilde \Phi ) = V^k (osp(1,2))  \quad \mbox{or} \quad \mbox{Im} (\widetilde \Phi ) = V_k (osp(1,2)) . $$
But, if $ \mbox{Im} (\widetilde \Phi ) = V^k (osp(1,2))  $, then the subalgebra generated by the embedding $sl(2)$ into $osp(1,2)$ must be universal affine vertex algebra $V^k(sl(2))$. A contradiction. So $ \mbox{Im} (\widetilde \Phi ) = V_k (osp(1,2))$, and first assertion holds.

(2)  By using relation
$$ x(-1) y - \omega_{sug} ^{ sl_2}  - \frac{1}{2} h(-2) = - \frac{p}{q} \omega_{p,\frac{p+q}{2}} $$
we see that $\omega_{p,\frac{p+q}{2}} \in V_{k} (osp(1,2)) $.
Since  $V_k(sl(2)) \subset L^{Vir} (d_{\frac{p+q}{2}, q}, 0) \otimes \Pi(0)$ we get that $\omega_{p, \frac{p+q}{2}}$ commutes with the action of $V_k(sl(2))$. The claim (2) follows.

\end{proof}

 \begin{remark} T. Creutzig and A. Linshaw  studied the coset $ Com(V_k (sl(2)), V_k(osp(1,2))$,  and proved in \cite[Theorem 8.2]{CL} that if $k$ is admissible, then the coset is isomorphic to a minimal  Virasoro vertex algebra. This can be also directly  proved  from     Theorem  \ref{exp-osp-adm}.
 \end{remark}

\subsection{Realization of $V^ {k} (osp(1,2)) $ at the critical level}

Let $M(0) ={\C} [b (-n) \ \vert n \ge 1]$ be the commutative vertex algebra generated by the field $b(z) = \sum_{n \le -1} b(n) z^{-n-1}$.

Let $\mbox{NS}_{cri}$ the infinite-dimensional Lie superalgebra with generators
$$C, T(n),  G^{cri} (n+\frac{1}{2}) \quad (n \in {\Z})$$ such that $T(n), C$ are in the center and
$$ \{ G^{cri} (r), G^{cri} (s) \} = 2 T(r+s) + \frac{r^2 - \tfrac{1}{4}}{3} \delta_{r+s, 0} C \quad (r,s \in \tfrac{1}{2} + {\Z}). $$

Let  $V^{ns} _{cri}$ be the  universal vertex superalgebra  associated to $\mbox{NS}_{cri}$ such that $C$ acts as scalar  $C=-3$.  $V^{ns} _{cri}$ is freely generated by odd field  $G^{cri} (z) = \sum_{n \in {\Z}} G^{cri} (n+\frac{3}{2}) z^{-n-2} $ and even vector $T (z) = \sum_{n \Z} T(n) z^{-n-2}$ such that $T$ is in central  and that the following $\lambda$--bracket relation holds:
$$ [G ^{cri} _{\lambda} G^{cri} ] =  2 T  -  \lambda ^2. $$

 $V^{ns} _{cri}$ can by realized as the vertex subalgebra of $F_2 \otimes M(0)$ generated by
$$  G^{cri}  =   b (-1) \Psi_2 (-\frac{1}{2}) +      \Psi_2 (-\frac{3}{2}), \  T = \frac{1}{2} ( b(-1) ^2  +  b(-2) )  .  $$

 By direct calculation we get that
 $$\omega_{1,2} =   T(-2) + G^{cri} (-\frac{3}{2}) \Psi (-\frac{1}{2})  + 2 \omega_F$$
 is a Virasoro vector of central charge  $c_{1,2} = -2$.

\begin{theorem} \label{homo-critical}
\item[(1)] Assume that $k =-3/2$. There exists a non-trivial homomorphism
$$ \overline \Phi : V^ {k} (osp(1,2)) \rightarrow  V^{ns} _{cri} \otimes F  \otimes     \Pi ^{1/2} (0)    $$
such that
 \bea\
e & \mapsto & e^{\tfrac{2}{k} (\mu - \nu) }, \nonumber  \\
h & \mapsto & 2 \mu(-1), \nonumber \\
f & \mapsto & \left[    (k+2)  \omega_{1,2}  -\nu(-1)^{2} -  (k+1) \nu(-2) \right] e^{-\tfrac{2}{k} (\mu - \nu)} \nonumber \\
x & \mapsto &   \sqrt{2}   \Psi  (-\frac{1}{2})    e^{\tfrac{1}{k} (\mu - \nu) } \nonumber \\
y & \mapsto & \sqrt{2}   \left[  - \frac{i }{2} G ^{cri} (-\frac{3}{2})  + \Psi  (-\frac{1}{2})  \nu (-1) + \frac{2k+1}{2}   \Psi  (-\frac{3}{2})  \right] e^{-\tfrac{1}{k} (\mu - \nu)} . \quad    \nonumber
\eea
\item[(2)]   $T    = \frac{1}{2} G^{cri} (-\frac{1}{2}) G^{cri} (-\frac{3}{2}) {\bf 1} $ is a central element of  $V^{k} (osp(1,2)) $.
\end{theorem}

 \begin{remark}
 In the same way as in \cite[Section 8.2]{KW04}  one can show that    for $k=-\tfrac{3}{2}$:
 $$\mathcal W^{k} (osp(1,2), f_{\theta}) \cong V^{ns} _{cri} .$$
 T. Arakawa proved in \cite{Ara-critical} that when $\g$ is a Lie algebra  and $f$ nilpotent element, then $$\mathfrak Z( \mathcal W^{-h^{\vee} } (\g, f) ) = \mathfrak Z (V^{-h^{\vee} } (\g) ). $$
 We believe  that  the results from \cite{Ara-critical}  hold for $\g=osp(1,2)$,  which  would prove the following:
 \begin{itemize}
 \item  $\mathfrak Z  (V^ {k} (osp(1,2)) ) \cong \mathfrak  Z( V^{ns} _{cri}) \cong M_T (0)$;
\item the homomorphism  $\overline \Phi $ from Theorem \ref{homo-critical} is injective.
\end{itemize}
   \end{remark}

 \section{ Proof of Theorem \ref{thm-reali-osp}}
 \label{realizacija-proof}

 We shall first prove an  important technical lemma.

\begin{lemma} \label{pomoc1}
Let $\bar y =  \left[   \Psi (-\frac{1}{2})  \nu (-1) + \frac{2k+1}{2}     \Psi (-\frac{3}{2})  \right] e^{-\tfrac{1}{k} (\mu - \nu)}$. We have:
\begin{itemize}
\item[(1)] $\bar y(2) \bar y =   - \frac{1}{4}  (2 k+ 1) ( 4 k+  5 )  e^{-\tfrac{2 }{k} (\mu - \nu)}, $
\item[(2)] $ \bar y(1) \bar y =    \frac{ (2k+1) (4 k + 5) }{4 k} (\mu(-1) - \nu (-1) )  e^{-\tfrac{2 }{k} (\mu - \nu)}   = -  \frac{ (2k+1) (4 k + 5) }{ 8  }  D   e^{-\tfrac{2 }{k} (\mu - \nu)},  $
\item[(3)] $\bar y(0) \bar y =\left( \frac{2k+1}{4} \Psi(-\tfrac{3}{2} ) \Psi (-\tfrac{1}{2}) + \nu(-1) ^2 + (k+1) \nu(-2)   -\frac{ (2k+1) (4k +5)}{4}  S_2(\frac{\nu-\mu}{k}) \right)  e^{-\tfrac{2 }{k} (\mu - \nu)},$
\end{itemize}
where $S_2(\gamma) = \frac{1}{2} ( \gamma(-1) ^2 + \gamma(-2))$.
\end{lemma}
\begin{proof}
 Let $\overline \tau = \sqrt{  \frac{-2}{k}}  \left( \Psi (-\frac{1}{2})  \nu (-1) +  (k+1)      \Psi (-\frac{3}{2})  \right){\bf 1},$ $b^{r }  =  e^{  \tfrac{r }{k} (\mu - \nu)} $. Then $\overline \tau$ generates a $N=1$ superconformal vertex algebra of central charge $\overline c =\frac{3 }{2k} (4 (k+1) ^2 + k)  $. We have
 $$\overline \omega = \frac{1}{2} \overline \tau  _0 \overline \tau  =-\frac{1}{ k} \left(  \nu(-1) ^2 + (k+1)  \nu(-2) \right)+ \omega_{fer}. $$

 $$ \bar y(p) = \sqrt{-\frac{k}{2}} (\tau_{-1}  b^{-})_p = \sqrt{-\frac{k}{2}}   \sum_{j=0}^{\infty} \tau_{-1-j} b^{-1} _{p+j} + b^{-1} _{-1-j +p} \tau_j. $$
 By  applying formulas
\bea  \tau _n b^{-1} & = &-\frac{1}{2}   \sqrt{  \frac{-2}{k}}   \delta_{n,0}    \Psi (-\frac{1}{2})  b^{-1} \qquad (n \ge 0), \nonumber \\
 \tau_j \tau_{-1}    b^{-1} &=& 0 \quad (j \ge 3), \nonumber \\
  \tau_2 \tau_{-1}    b^{-1} &=&\left( \frac {4 (k+1) ^2 + k}  {k} - \frac{2 k+3}{2 k}  \right)b^{-1} = \frac{ 8 (k+1) ^2 -3}{2k} b^{-1}\nonumber \\
\tau_1 \tau_{-1}    b^{-1} &=& \frac{2}{k} \nu(-1) b^{-1}  \nonumber \\
\tau_0 \tau_{-1}    b^{-1} &=& -\frac{2}{k} \left( \nu (-1) ^2  + k \nu(-2) \right)  b^{-1} +   \Psi(-\frac{3}{2}) \Psi(-\frac{1}{2}) b^{-1} \nonumber \\
&& -\frac{1}{k} \nu(-2) b^{-1}-\frac{2k+1}{2k}  \Psi(-\frac{3}{2}) \Psi(-\frac{1}{2}) b^{-1} \nonumber  \\
&=& -\frac{2}{k} \left( \nu (-1) ^2  + \frac{2 k+1}{2}  \nu(-2) \right)  b^{-1} -\frac{1}{2k}  \Psi(-\frac{3}{2}) \Psi(-\frac{1}{2}) b^{-1}  \nonumber
 \eea
%
 we get
\bea
\bar y(2) \bar y
&=&  -\frac{k}{2} \left(   b^{-1} _{-1} \tau_2 \tau _{-1} b^{-1}  +  b^{-1} _{0} \tau_1 \tau _{-1}  b^{-1}  +   b^{-1} _{1} \tau_0 \tau _{-1}  b^{-1}   \right) \nonumber \\
&=&   -\frac{k}{2} \left(  \frac{  4 (k+1) ^2 + k  }{k} -\frac{2 k+3}{2 k}  +\frac{1}{k} - \frac{ k+1}{k}  \right)  b^{-2}\nonumber \\
&=&  - \frac{1}{4}  (2 k+ 1) ( 4 k+  5 )  b^{-2}. \nonumber  \eea
\bea
\bar y(1) \bar y & = &  -\frac{k}{2}  \left(   b^{-1} _{-2} \tau_2 \tau _{-1} b^{-1}  +  b^{-1} _{-1} \tau_1 \tau _{-1}  b^{-1}  +   b^{-1} _{0} \tau_0 \tau _{-1}  b^{-1}   \right) \nonumber \\
&=& -\frac{k}{2} ( -\frac{ 8 (k+1) ^2 -3}{2k^2 }  (\mu(-1)  -\nu(-1) )  b^{-2} \nonumber  \\
&& +\frac{2}{k} \nu(-1) b^{-2} - \frac{1}{k^2 } (\mu(-1)-\nu(-1)) b^{-2} \nonumber  \\
&& + \frac{k+1}{k^2} (\mu(-1) -\nu(-1)) b^{-2} -\frac{2}{k} \nu(-1) b^{-2}  ) \nonumber \\
&=& \frac{ ( 4k +5) (2k+1)} {4 k}  (\mu(-1) -\nu(-1)) b^{-2} = -  \frac{ (2k+1) (4 k + 5) }{ 8  }  D   e^{-\tfrac{2 }{k} (\mu - \nu)}.\nonumber \eea
\bea
\bar y(0) \bar y & = &  -\frac{k}{2}  \left(   b^{-1} _{-3} \tau_2 \tau _{-1} b^{-1}  +  b^{-1} _{-2} \tau_1 \tau _{-1}  b^{-1}  +   b^{-1} _{-1} \tau_0 \tau _{-1}  b^{-1} +     \tau_{-1} b^{-1} _{0}   \tau _{-1}  b^{-1}\right) \nonumber \\
&=&  -\frac{k}{2} ( \frac{ 8 (k+1) ^2 -3}{2k }  S_2(\frac{\nu-\mu}{k})  b^{-2} -\frac{2}{k^2} \nu(-1) (\mu(-1)-\nu(-1)) b^{-2} \nonumber \\
&& \quad +\frac{1}{k}  S_2(\frac{\nu-\mu}{k})  b^{-2}   -\frac{2}{k} \left( \nu (-1) ^2  + \frac{2 k+1}{2}  \nu(-2) \right)  b^{-2} -\frac{1}{2k}  \Psi(-\frac{3}{2}) \Psi(-\frac{1}{2}) b^{-2}  \nonumber \\
&& \quad -\frac{k+1}{k}  S_2(\frac{\nu-\mu}{k})  b^{-2}  + \frac{2}{k^2} \nu(-1)  (\mu(-1) -\nu(-1)) b^{-2} \nonumber \\
&& \quad -\frac{1}{k} \left( \nu(-2) + k  \Psi(-\frac{3}{2}) \Psi(-\frac{1}{2}) b^{-2}\right) ) \nonumber \\
&=& -\frac{ (2k+1) (4k +5)}{4}  S_2(\frac{\nu-\mu}{k})  b^{-2} +   \frac{2k+1}{4}  \Psi(-\frac{3}{2}) \Psi(-\frac{1}{2}) b^{-2}  \nonumber \\
&& + ( \nu(-1) ^2 + (k+1) \nu(-2)  ) b ^{-2} \nonumber
 \eea
The proof follows.
\end{proof}

\subsection{ Proof of Theorem \ref{thm-reali-osp} }
First we notice that
$
c_{p,q} = -\frac{3}{2} \frac{ (4k +5) (2k+1) }{2k +3}.
$


Assume that  $k \ne -2$. Since   $$e,f,h \in V^{Vir} (d_{\frac{p+q}{2} , q}, 0) \otimes \Pi(0) \subset V^{ns} (c_{p, q}, 0) \otimes  F \otimes \Pi ^{1/2} (0), $$
then Proposition \ref{prop-sl2-1} implies  that vectors  $e,f, h$   generate a vertex subalgebra  of  $V^{ns} (c_{p, q}, 0) \otimes  F \otimes \Pi ^{1/2} (0)$ isomorphic to $V^k(sl(2))$.  In the case $k=-2$,  vector $\Omega_{p,q}$ generates a commutative vertex algebra isomorphic to $M_T(0)$,  and therefore Proposition \ref{prop-sl2--crit}  implies that $e,f,h$ generate a quotient of $V^{k}(sl(2))$.

Next we need to prove that for $n \ge 0$ the following relations hold:
\bea
&& h(n) x = \delta_{n,0} x, \ e(n) x = 0, \ x(n) f = \delta_{n,0} y  \label{osp-1} \\
&& h(n) y = -\delta_{n,0} y, \ e(n) y = -\delta_{n,0} x, \ f(n) y = 0 \label{osp-2} \\
&& x(n) x = 2 \delta_{n,0} e, \ y(n)y = -2 \delta_{n,0} y, \ x(0) y = h, \ x(1) y =   2 k {\bf 1}. \label{osp-3}
\eea

 Let us first prove that $x(n) f = \delta_{n,0} y$ for $n \ge 0$. Clearly $x(n) f = 0$ for $n \ge 2$. We have:
 \bea
 x(1) f& = &\sqrt{2} \left(\Psi(\frac{3}{2}) \Omega_{p,q}  + \frac{2k +1}{4} \Psi(-\frac{1}{2}) \right) e^{-\tfrac{1}{k} (\mu - \nu)}  \nonumber \\
 &=& \sqrt{2} \left(-\frac{2k+1}{4} \Psi(\frac{3}{2}) \Psi(-\frac{3}{2}) \Psi(-\frac{1}{2})   + \frac{2k +1}{4} \Psi(-\frac{1}{2}) \right) e^{-\tfrac{1}{k} (\mu - \nu)}  = 0, \nonumber  \\
 x(0) f & = &\sqrt{2} \left(\Psi(\frac{1}{2}) \Omega_{p,q}  + \frac{2k +1}{4} \Psi(-\frac{3}{2}) \right) e^{-\tfrac{1}{k} (\mu - \nu)}  \nonumber \\
 & & +  \sqrt{2}  \Psi(\frac{3}{2}) \Omega_{p,q}   \frac{1}{k} (\mu-\nu) (-1)   e^{-\tfrac{1}{k} (\mu - \nu)} \nonumber \\
 & & + \sqrt{2}  \frac{2k+1}{4}\Psi(-\frac{1}{2})   \frac{1}{k} (\mu-\nu) (-1)   e^{-\tfrac{1}{k} (\mu - \nu)}   + \sqrt{2} \Psi(-\frac{1}{2}) \nu(-1)  e^{-\tfrac{1}{k} (\mu - \nu)}  \nonumber \\
 &= & \sqrt{2}   \left[ - \frac{ \sqrt{-2k- 3} }{2 }    G (-3/2) +   \Psi (-\frac{1}{2})  \nu (-1) + \frac{2k+1}{2}   \Psi (-\frac{3}{2})  \right] e^{-\tfrac{1}{k} (\mu - \nu)}  = y. \nonumber
 \eea
 By using an easy calculation we get:
 \bea
 x(1) y & = &  2 (  \frac{2k+1}{2} -\frac{1}{2}) {\bf 1}  = 2k {\bf 1},  \nonumber \\
 x(0) y & = &  2k \frac{1}{k} (\mu-\nu) (-1) + 2\nu(-1) = 2\mu(-1) = h,  \nonumber  \\
 e(0) y & = & \delta_{n,0} x, \nonumber \\
 x(n) x &= &  2 \delta_{n,0} e . \nonumber
 \eea
Finally, we  will check relation $y(n) y = -2 \delta_{n,0}$.
Clearly, $y (n) y= 0$ for $n\ge 3$.
For the cases $n=0,1,2$ we need to use Lemma \ref{pomoc1}. We have:
\bea
y(2) y
&=&  ( - \frac{1}{2}  (2 k+ 1) ( 4 k+  5 )   - \frac{2k+3}{3} c_{p,q}  )  e^{-\tfrac{2 }{k} (\mu - \nu)}  = 0, \nonumber \\
y(1) y
& = &  \left(  \frac{ (2k+1) (4 k + 5) }{2 k}    + \frac{2k+3}{3 k }    c_{p,q}   \right)   (\mu(-1) - \nu (-1) )  e^{-\tfrac{2 }{k} (\mu - \nu)}   = 0, \nonumber \\
y(0) y
&=&   \left(  -\frac{ (2k+1) (4k +5)}{4}  -\frac{2k+3} {3} c_{p,q}\right)  S_2(\frac{\nu-\mu}{k})   e^{-\tfrac{2 }{k} (\mu - \nu)} \nonumber \\
&& +   \left(   -(2k +3) \omega_{p,q} + ( 2k+1)\omega_{F} -i \sqrt{2k+3} G(-\frac{3}{2}) \Psi(-\frac{1}{2})  \right)   e^{-\tfrac{2 }{k} (\mu - \nu)}  \nonumber \\
&&  +   \left( 2(\nu (-1) ^2 + (k+1) \nu(-2))   \right )      e^{-\tfrac{2 }{k} (\mu - \nu)} \nonumber \\
&=&- 2  \left(  (k+2) \omega_{p,\frac{p+q}{2}} -   (\nu (-1) ^2 + (k+1) \nu(-2))    \right)   e^{-\tfrac{2 }{k} (\mu - \nu)}  = - 2 f. \nonumber
\eea
In this way we have checked relations (\ref{osp-1})-(\ref{osp-3}).
This finishes the proof of Theorem. \qed

 \section{Example: weight  and Whitaker modules for $k =-5/4$ }

  As we have seen in previous sections (see also \cite{ALZ-2016},  \cite{AR-2017}, \cite{A-2017}) for the analysis  of weight, Whittaker  and logarithmic modules, the explicit free-field  realization is very useful.

 The realization of $V_{k} (osp(1,2))$ is simpler  in the cases when $L^{ns} (c_{p,q}, 0)$ is a $1$--dimensional vertex algebra, and therefore $ V_{k} (osp(1,2))$ can be realizaed on the vertex algebra $F \otimes  \Pi ^{1/2} (0)$. This happens only in the cases $k =-\frac{1}{2}$ and $k = -\frac{5}{4}$. In the case $k=-\frac{1}{2}$, $ V_{k} (osp(1,2))$  can be realized on  the tensor product of the Weyl vertex algebra $W$ with the fermionic vertex algebra $F$ of central chargce $c=1/2$. But this is essentially known in the literature, as a special case of the realization of $V_{-1/2} (osp(1, 2n))$  (cf. \cite{FF}).

 In this section we specialize  our realization to the case $k = -5/4$.  We get a realization of the  vertex algebra $V_k(osp(1,2)$, which was investigated by D. Ridout,   J. Snadden and S. Wood \cite{RSW} by using different methods.  It is  also important to notice that the  vertex algebra  $ V_{k} (osp(1,2))$ is a simple current extension of $V_k(sl(2))$:
 $$ V_{-\frac{5}{4} } (osp(1,2)) =  L _{A_1} (-\frac{5}{4} \Lambda_0) +  L _{A_1} (-\frac{9}{4} \Lambda_0  + \Lambda_1),$$
 which can be also proved from our realization.
Then $k + \frac{3}{2} = \frac{p}{2 q}$ for $p = 2$, $q = 4$. Since $c_{p,q} = 0$, we have that $L^{ns} (c_{p,q}, 0)$ is a $1$--dimensional vertex algebra.

 \begin{remark}
 Note that $k = -h^{\vee} / 6-1$ is a  collapsing level   for $\g=osp(1,2)$ \cite{AKMPP-JA}. In this case we have the realization inside the free field algebra  $F \otimes   \Pi ^{1/2} (0)$ without any $\mathcal W$-algebra.
 \end{remark}
 We have the following realization of $V_{k} (osp(2,1))$.

\begin{corollary}   \label{nivo54} Assume that   $k = -\frac{5}{4}$.
 \item[(1)] There exists  a non-trivial vertex superalgebra homomorphism
$$ \overline \Phi : V_{k} (osp(1,2)) \rightarrow   F \otimes   \Pi ^{1/2} (0) $$
such that
 \bea\
e & \mapsto & e^{\tfrac{2}{k} (\mu - \nu) }, \nonumber  \\
h & \mapsto & 2 \mu(-1), \nonumber \\
f & \mapsto & \left[    (k+2) \omega_{3,4}    -\nu(-1)^{2} -  (k+1) \nu(-2) \right] e^{-\tfrac{2}{k} (\mu - \nu)} \nonumber \\
x & \mapsto &   \sqrt{2}   \Psi (-\frac{1}{2})    e^{\tfrac{1}{k} (\mu - \nu) } \nonumber \\
y & \mapsto & \sqrt{2}   \left[    \Psi (-\frac{1}{2})  \nu (-1) + \frac{2k+1}{2}   \Psi (-\frac{3}{2})  \right] e^{-\tfrac{1}{k} (\mu - \nu)},   \quad    \nonumber
\eea
where $\omega_{3,4} = \frac{1}{2} \Psi (-\frac{3}{2}) \Psi (-\frac{1}{2}) {\bf 1}$.

\item[(2)] Assume that $ U$  (resp. $U^{tw}$) is any  untwisted (resp. $g$-twisted)  $\Pi^{1/2} (0)$--module. Then
\begin{itemize}
\item $ F \otimes U$  and $M ^{\pm} \otimes U^{tw}$ are untwisted $V_k (osp(1,2))$--modules.
\item $F \otimes U^{tw}$ and $M^{\pm} \otimes U$ are Ramond twisted  $V_k (osp(1,2))$--modules.
\end{itemize}
\end{corollary}

A classification of  irreducible untwisted  and twisted  $V_{k} (osp(2,1))$--modules were obtained  \cite[Theorem 9]{RSW}  by using  Zhu's algebra approach. All representations can be constructed using our free-field realization. Maybe most interesting examples are relaxed highest weight   $V_{k} (osp(2,1))$--modules. We shall consider here only Neveu-Schwarz sector, i.e, non-twisted  $V_{k} (osp(2,1))$--modules.

  %

Consider  the $\sigma \otimes g$--twisted module $F \otimes   \Pi ^{1/2} (0)$--module
$\mathcal F ^{\lambda} :=  M ^{\pm} \otimes \Pi ^{ (1/2)} _ {(-1)} (\lambda)$
 for  $\lambda \in {\C}$.
 Then $    \mathcal F ^{\lambda} $ is an untwisted $V_k (osp(1,2))$--module.
As in Section  \ref{relaxed-realization} we define
$ E_{1,2}  ^{\lambda} = {\bf 1} ^{\pm}  \otimes e^{- \mu + \lambda \frac{2}{k} (\mu - \nu)}$.
Then the action of $osp(1,2)$ is given by
 \bea
 e(0) E_{1,2}   ^{\lambda}  &=& E_{1,2}  ^{\lambda +1}, \nonumber \\
 h(0)  E_{1,2}   ^{\lambda} &=& (-k + 2 \lambda )E_{1,2}  ^{\lambda}, \nonumber \\
 f(0) E _{1,2} ^{\lambda}
  &=& \left( \frac{ 1 }{  16 } - (\lambda + \frac{ 1 }{8  } ) ^2 \right)   E_{1,2} ^{\lambda-1}  = ( \frac{3}{8} + \lambda )  (\frac{1}{8} - \lambda )   E_{1,2} ^{\lambda-1}  \nonumber \\
  x(0)  E _{1,2} ^{\lambda}  & =  & \pm E_{1,2} ^{\lambda + \frac{1}{2} } \nonumber     \\
  y(0) E _{1,2} ^{\lambda}  & =  &  \mp    ( \frac{3}{8} + \lambda ) E_{1,2} ^{\lambda  -\frac{1}{2} } .\nonumber
 \eea
 Moreover, we have
 $$L_{sug} (n) E_{1,2}   ^{\lambda}   = - \frac{1}{4} \delta_{n,0} E_{1,2}   ^{\lambda}  \quad (n \ge 0). $$

\begin{theorem} Assume that $\lambda \notin \frac{1}{8} +  \frac{1}{2}  {\Z}$.
Then $\mathcal F ^{\lambda}$ is an irreducible ${\Z}_{\ge 0}$ --graded $V_{k} (osp(1,2))$-module whose character is
\bea \mbox{ch} [\mathcal F ^{\lambda }]  (q,z)  &=& \mbox{Tr} _{ \mathcal F ^{\lambda}} q^ {L_{sug} (0) - c/24} z^ { h(0)}
= z^{ 2 \lambda -k }\frac{f_2(\tau) \delta(z)} {{ \sqrt 2} \eta(\tau) ^2},  \nonumber \eea
where
$f_2 (\tau)= \sqrt{2}  q^{ \frac{1}{24 } }   \prod_{n =1} ^{\infty} (1+ q^n)$.
(In the terminology of \cite{RSW}, $\mathcal F ^{\lambda}$ corresponds to ${\mathcal C}_{\Lambda,0}$ where $\Lambda = 2 \lambda + \frac{5}{4}$).
\end{theorem}
\begin{proof} Note first that $\mathcal F ^{\lambda}$ is ${\Z}_{\ge 0}$--graded  and that its lowest component is
$ \mathcal F ^{\lambda} (0) = \mbox{span}_{\C} \{ E _{1,2} ^{\lambda + i }, \ \  i \in \tfrac{1}{2} {\Z} \}$. The $osp(1,2)$--action obtained above implies that $ \mathcal F ^{\lambda} (0)$ is irreducible for  $\lambda \notin \frac{1}{8} +  \frac{1}{2}  {\Z}$. By using realization, we see that as $V_k(sl(2))$--module we have
$    \mathcal F ^{\lambda} = \mathcal E_{1,2} ^{\lambda} \oplus \mathcal E_{1,2} ^{\lambda + 1/2}$,
where
$ \mathcal E_{1,2} ^{r } = L^{Vir} (d_{3,4}, \frac{1}{16} ) \otimes \Pi_{-1} (r)$.
By Proposition  \ref{maximal},  $ \mathcal E_{1,2} ^{r } $ is irreducible for $r  \notin \frac{1}{8} +  \frac{1}{2}  {\Z}$. Therefore  $\mathcal F ^{\lambda}$ is a direct sum of two irreducible
 $V_k(sl(2))$--modules, which easily gives irreducibility result  since $V_{k} (osp(1,2))$ is a simple current extension of $V_k(sl(2))$.   The character formula follows directly  from the realization, character formula for $\Pi_{(-1)} ^{1/2} (\lambda)$  (cf. Proposition  \ref{char-pi0}) and   (\ref{weber}):
 \bea && \mbox{ch} [\mathcal F ^{\lambda }]  (q,z) = \mbox{ch}[M](q) \mbox{ch}[\Pi_{(-1)} ^{1/2} (\lambda)](q) =  z^{ 2 \lambda -k }\frac{f_2(\tau) \delta(z)} {{ \sqrt 2} \eta(\tau) ^2} \eea
\end{proof}

 We also have the following result on the irreducibility of some  Whittaker modules.

 \begin{corollary} We have:
 $ M ^{\pm} \otimes \Pi_{\lambda}$ is irreducible $V_{k} (osp(1,2))$--module.
 \end{corollary}
 \begin{proof}
 $ M ^{\pm} \otimes \Pi_{\lambda}$  is a $V_{k} (osp(1,2))$--module by Corollary \ref{nivo54} (2). The irreducibility follows from the fact that
  $M ^{\pm} \otimes \Pi_{\lambda}$  is, as  a $V_{k} (sl(2))$--module,  isomorphic to  the Whittaker module $L^{Vir} (d_{3,4}, \frac{1}{16} ) \otimes \Pi_{\lambda}$, which is irreducible.
 \end{proof}

\begin{remark}  \label{generalization-osp(1,2)} A generalization of modules constructed above is as follows. Let  $L^{\mathcal R} (c_, h) ^{\pm}$  be the  irreducible Ramond twisted modules for  the simple $N=1$ Neveu-Schwarz vertex algebra $L^{ns} (c, 0)$ (cf. \cite{AdM-sigma},  \cite{M-crelle}, \cite{FMRW}).
For an arbitrary admissible level $k$, we have  the following family of ${\Z}_{\ge 0}$--graded relaxed and Whittaker  $V_{k} (osp(1,2))$--modules:
$$ L^{\mathcal R} (c_{p,q}, h) ^{\pm} \otimes M^{\pm} \otimes \Pi_{(-1)} ^{ 1/2} (\lambda), \quad   L^{\mathcal R} (c_{p,q}, h) ^{\pm} \otimes M^{\pm} \otimes \Pi_{\lambda}. $$
The irreducibility of modules $ L^{\mathcal R} (c_{p,q}, h) ^{\pm} \otimes M^{\pm} \otimes \Pi_{-1} ^{ 1/2} (\lambda)$ can be proved by using character formulas for irreducible relaxed $\widehat{osp(1,2)}$--modules from \cite[Theorem 8.2]{KR}.

Let $ h  =  \Delta_{r,s} $ be  given by formula (8.11) in \cite{KR}. Using Proposition  \ref{char-pi0}, we see that the character of the module  $ L^{\mathcal R} (c_{p,q}, h) ^{\pm} \otimes M^{\pm} \otimes \Pi_{(-1)} ^{ 1/2} (\lambda)$
is given by
\bea &&  \mbox{ch} [L^{\mathcal R} (c_{p,q}, h) ^{\pm}](q)  \mbox{ch} [M^{\pm}](q) \mbox{ch}[ \Pi_{(-1)} ^{ 1/2} (\lambda)] (q,z) \nonumber \\ &=& \mbox{ch} [L^{\mathcal R} (c_{p,q}, h) ^{\pm}](q) f_2(\tau)  \frac{z^{ 2 \lambda -k }}{\sqrt{2} \eta(\tau) ^2 }   \delta (z ), \nonumber  \eea
which coincides with the character of the admissible relaxed $V_k(osp(1,2))$--module $^{\text{\tiny NS}}\mathcal E _{2 \lambda-k, q_{r,s}}$ in \cite[Theorem 8.2]{KR}.

\end{remark}

\end{document}